\documentclass{amsart}
\usepackage[english]{babel}
\usepackage{amssymb,enumerate,amsmath,txfonts}
 \usepackage{graphicx}
\usepackage[colorlinks=true,linkcolor=blue,citecolor=blue,urlcolor=blue]{hyperref}
\numberwithin{equation}{section}

\newtheorem{theo}{Theorem}[section]

\newtheorem{lem}[theo]{Lemma}
\newtheorem{cor}[theo]{Corollary}
\newtheorem{claim}[theo]{Claim}
\newtheorem{conjecture}[theo]{Conjecture}

\newtheorem*{mainthm}{Main Theorem}
\theoremstyle{remark}

\newtheorem{rem}[theo]{Remark}
\newtheorem{example}[theo]{Example}

\renewcommand{\(}{\left(}
\renewcommand{\)}{\right)}

\newcommand{\R}{\mathbb{R}}

\renewcommand{\a}{\alpha}

\newcommand{\g}{\gamma}
\renewcommand{\d}{\delta}
\newcommand{\e}{\epsilon}
\renewcommand{\k}{\kappa}
\renewcommand{\l}{\lambda}

\newcommand{\D}{\Delta}

\renewcommand{\t}{\theta}

\renewcommand{\L}{\Lambda}

\newcommand{\ra}{\rightarrow}

\newcommand{\mrm}{\mathrm}

\newcommand{\Vol}{\operatorname{Vol}}

\newcommand{\Ric}{\operatorname{Ric}}

\newcommand{\inj}{\operatorname{inj}}

\newcommand{\divv}{\operatorname{div }}
\newcommand{\Hess}{\operatorname{Hess }}
\newcommand{\conj}{\operatorname{conj}}
\newcommand{\conv}{\operatorname{conv}}

\begin{document}
\title[Recognizing shape via 1st eigenvalue and mean curvature]{Recognizing shape via 1st eigenvalue, mean curvature and upper curvature bound}
\author{Yingxiang Hu, Shicheng Xu}

\address{Yau Mathematical Sciences Center, Tsinghua University, Beijing, China}
\email{huyingxiang@mail.tsinghua.edu.cn}

\address{School of Mathematical Sciences, Capital Normal University, Beijing, China \newline
\indent Academy for Multidisciplinary Studies, Capital Normal University, Beijing, China}
\email{shichengxu@gmail.com}
\date{\today}
\begin{abstract}
	Let $M^n$ be a closed immersed hypersurface lying in a contractible ball $B(p,R)$ of the ambient $(n+1)$-manifold $N^{n+1}$. We prove that, by pinching Heintze-Reilly's inequality via sectional curvature upper bound of $B(p,R)$, 1st eigenvalue and mean curvature of $M$, not only $M$ is Hausdorff close to a geodesic sphere $S(p_0,R_0)$ in $N$, but also the ``enclosed'' ball $B(p_0,R_0)$ is close to be of constant curvature, provided with a uniform control on the volume and mean curvature of $M$. We raise a conjecture for $M$ to be a diffeomorphic sphere, and give some positive partial answer.
\end{abstract}
{\maketitle}

\section{Introduction}

The isoperimetric inequality in the Euclidean plane $\mathbb R^2$ has a long history, and has many generalizations both on Riemannian surfaces and higher dimensional manifolds (e.g. \cite{Osserman1978}, \cite{Milman-Schechtman}). One of those such that ``equality implies rigidity'' was founded around 1950's, as follows.

Let $D$ be a simply-connected domain on a surface $S$, $L$ be the length of boundary $\partial D$ and $A$ the area of $D$. Then (\cite{Alexandrov1945,Alexandrov-Strel'tsov1965}, \cite{Huber1954}, \cite{Barbosa-doCarmo1978}, cf.  \cite{Osserman1978} and references therein)
\begin{align}
\label{isoper2}
L^2 &\ge 4\pi A-(\sup_D K)A^2
\end{align}
where $K(p)$ is the Gauss curvature. If equality holds in (\ref{isoper2}), then $D$ is a geodesic disk in the space form of constant curvature $K$ (cf. \cite{Barbosa-doCarmo1978}, \cite{Chavel-Feldman80}).

Since then, however, not only few natural generalization of (\ref{isoper2}) with similar rigidity are known on higher dimensional manifolds, but also other rigidity phenomena with respect to the upper curvature bound are rarely studied.

In contrast, nowadays rigidity results and their quantitative version (e.g. \cite{Colding1996-1,Colding1996-2}, \cite{Cheeger-Colding1996},  \cite{Petersen1999}, \cite{Chen-Rong-Xu2016,Chen-Rong-Xu2018}, \cite{Miao-Wang2016}, \cite{Schroeder-Strake1989}, etc.) under curvature bounded from below have been extensively studied. They provide fundamental tools in study of Riemannian manifolds and their limit geometry under Gromov-Hausdorff topology.

In this paper, we prove a quantitative rigidity for a domain (resp. an immersed closed hypersurface) in a contractible neighborhood on a complete Riemannian manifold to be a geodesic ball (resp. geodesic sphere) of constant curvature $\d$, where $\d$ is \emph{upper curvature bound} of ambient space.

Our starting point is an observation on Heintze's result \cite{Heintze1988}. An open ball $B(p,R)$ is called \emph{geodesic contractible} if for any point $x\in B(p,R)$, there is a unique radial minimal geodesic from $p$ to $x$; i.e., $R$ is no more than injectivity radius of $p$.

\begin{theo}\label{main-theo-A} Let $M^n$ be an immersed, oriented and connected compact hypersurface without boundary in a geodesic contractible ball $B(p,R)$ of $N^{n+1}$, where the sectional curvature of $N$, $K_N \leq \d$ for some $\d\in \R$. Let $|M|$ denote $M$'s volume and $H$ the mean curvature. If the first non-zero eigenvalue $\lambda_1(M)$ of Laplace-Beltrami operator on $M$ satisfies one of the following conditions,
	\begin{enumerate}[(i)]
		\item $\d\geq 0$, $R\le\frac{\pi}{4\sqrt{\d}}\in (0,\infty]$ and
		\begin{align}\label{1steigeneq-1}
		\l_1(M)= n\d+\frac{n}{|M|}\int_{M}|H|^2,
		\end{align}
		\item $\d<0$ and
		\begin{align}\label{1steigeneq-2}
		\l_1(M)= n\d+n\max_{M}|H|^2,
		\end{align}
	\end{enumerate}
	then $M$ is an embedded geodesic sphere, and the enclosed ball is of constant curvature $\d$.
\end{theo}

Recall that in \cite[Theorem 2.1]{Heintze1988} Heintze proved that a submanfold $M^n$ (including higher codimension) lying a convex ball $B(p,R)$ of $N^{n+m}$ satisfies
\begin{align}\label{1steigen-ineq}
\l_1(M)&\leq \begin{cases}
n\d+\frac{n}{|M|}\int_{M}|H|^2, &\text{for $\d\geq 0$ and $R\le\frac{\pi}{4\sqrt{\d}}\in (0,\infty]$,}\\
n\d+n\max_{M}|H|^2, &\text{for $\d<0$.}
\end{cases}
\end{align}
Equality holds in \eqref{1steigen-ineq} if and only if $M$ is minimally immersed in some geodesic sphere of $N$.

The observation in Theorem \ref{main-theo-A} is that, if equality holds in \eqref{1steigen-ineq}, then the norm of Jacobi fields along radial geodesic starting from spherical center to $M$ is that of constant curvature $\d$. Moreover, the convexity of $B(p,R)$ can be weakened to be geodesic contractible; see \S \ref{sec:2} for a complete proof of Theorem \ref{main-theo-A}.

An interesting consequence of Theorem \ref{main-theo-A} is that, any interior perturbation (no matter large or small) of a ball of constant curvature $\d\le 0$  has to raise up interior curvature. Such fact can be seen from Gauss-Bonnet theorem when $n=1$. But in high dimension it is hard to see without using (\ref{1steigen-ineq}) that involves upper curvature bound.

\begin{rem}
	In some sense, the inequality \eqref{1steigen-ineq} can be viewed as a high-order generalization of \eqref{isoper2}.
	
	Indeed, as one of a series of inequalities for $k$-th mean curvature, \eqref{1steigen-ineq} was first proved by Reilly \cite{Reilly1977} for submanifolds in $\mathbb{R}^{n+m}$ ($m\ge 1$), and \eqref{1steigen-ineq} corresponds to the case of $k=1$.
	
	According to \cite[Corollaries 1, 2]{Reilly1977}, the $0$-mean curvature $=1$, and for an embedded hypersurface $M$ enclosing domain $\Omega$, Reilly's inequality degenerates to $$n |M|^2\ge (n+1)^2  \lambda_1(M) \cdot \operatorname{vol}(\Omega)^2,$$ which coincides with the isoperimetric inequality (\ref{isoper2}) on $\mathbb R^2$.
\end{rem}

Our main result is a quantitative version of Theorem \ref{main-theo-A} via pinching (\ref{1steigeneq-2}).

As working for a class of manifolds, certain uniform geometric bounds are usually required. We will work under the assumptions that ambient space $N$ admits a bounded sectional curvature $\mu \leq K_N\leq \d$, the volume and mean curvature of immersed hypersurface $M$ satisfy the following rescaling invariant bound
\begin{equation}\label{rescaling-inv-non-local-col}
|M|^\frac{1}{n}\|H\|_\infty\leq A.
\end{equation}

By Lemma \ref{lem-local-ball-volume-estimate} below, volume of extrinsic ball on $M$ admits a uniform lower bound,
\begin{equation}\label{relative-non-collapsing}
|B(x_0,r)\cap M|\ge C(n,\d, A)r^n, 
\end{equation}
for any $x_0\in M$ and $0<r\le \min\left\{\operatorname{diam}_N(M), \frac{\pi}{2\sqrt{\d}}\right\}$.
Hence we may view (\ref{rescaling-inv-non-local-col}) as an extrinsically non-collapsing condition.

Let $\|H\|_\infty=\max_{M}|H|$ be the $L^\infty$-norm of mean curvature vector of $M$. Let $s_\delta$ be the usual $\delta$-sine function (see (\ref{2.1}) below) and $s_\d^{-1}$ its inverse function, respectively. Let $\omega_n$ be the volume of unit sphere in $\mathbb R^{n+1}$. Throughout the paper we view $\frac{\pi}{\sqrt{\d}}=\infty$ for $\d\le 0$, and use  $\varkappa(\epsilon\,|\,A,\cdots)$ to denote a positive function on $\epsilon, A,\cdots$ that converges to $0$ as $\epsilon\to 0$ with  other quantities $A,\cdots$ fixed.

\begin{mainthm}
	Let $n$ be an integer $\ge 2$, $0<R< \infty$, and let $N^{n+1}$ be a complete Riemannian manifold with $\mu \leq K_N\leq \d$. Let $M^n$ be an immersed closed, connected and oriented hypersurface in a geodesic contractible ball $B(p,R)$. If $\d>0$, we further assume
	\begin{equation}\label{location-size}
	R\le \frac{\pi}{8\sqrt{\d}},\quad |M|\le \omega_n s_\d^n\(\frac{\pi}{4\sqrt{\d}}\).
	\end{equation}
	If $M$ satisfies (\ref{rescaling-inv-non-local-col}) for $A=A_1>0$, and
	\begin{align}\label{pinching-condition-1}
	n(\d+\|H\|_\infty^2) \leq \l_1(M)(1+\e),
	\end{align}
	holds with $0\le \e<\e_0(A_1,R,\d,\mu,n)$,
	then
	\begin{enumerate}
		\item[(M1)]
		$M$ is $C_1\e^\frac{1}{2(2n+1)} s_\d(R_0)$-Hausdorff-close to a geodesic sphere $S(p_0,R_0)$, where $R_0=s_\d^{-1}(\frac{1}{\sqrt{\d+\|H\|_{\infty}^2}})$ and
		$C_1=C_1(n,\d,\mu,R,A_1)$ is a positive constant;
		\item[(M2)]
		$B(p_0,R_0)$ is $\varkappa(\epsilon\,|\,A_1,R,\d,\mu,n,\alpha)$ $C^{1,\alpha}$-close to a ball of constant curvature $\delta$ for any $0\le \alpha<1$.
	\end{enumerate}
\end{mainthm}

From the proof, $p_0$ is center of mass of $M$ in $N$ with respect to an appropriate variation of distance function; see \S \ref{sec:2.4}.

The conclusion of Main Theorem is only known before in space forms (\cite{Colbois-Grosjean2007}, \cite{Aubry-Grosjean},  \cite{Hu-Xu2017}).
(M2) reveals a substantially new phenomena on Riemannian manifolds; a contractible domain can be recognized to be a ball of almost constant curvature by ``hearing'' the 1st eigenvalue of its boundary $M$ (or any immersed hypersurface $M$ close to its boundary), maximum of $M$'s mean curvature and interior curvature's upper bound.

Due to that $\l_1(M)\le n\d+\frac{n}{|M|}\int_{M}|H|^2$ is known to be true only for $\d\ge 0$ \cite{Heintze1988} or on a hyperbolic space for $n\ge 2$ \cite{Soufi-Ilias1992} (it fails for $n=1$ if $N^2$ is hyperbolic, cf. \cite{Heintze1988}), the pinching condition (\ref{pinching-condition-1}) is the best one can generally expect at present on Riemannian manifolds.

Recall that by Schoen-Yau's positive mass theorem \cite{Schoen-Yau1979,Schoen-Yau2017}, any perturbation $g_t$ of the canonical metric $g_0$ in a ball of $\mathbb R^n$ has to lower the infimum of scalar curvature to be negative, unless $g_t$ coincides with $g_0$. 
Combining with Theorem \ref{main-theo-A} and Main Theorem, we obtain the following corollary.

\begin{cor}\label{cor-main}
	Any non-trivial interior metric perturbation of a bounded Euclidean domain $\Omega$ must both raise up supremum of sectional curvature and lower the infimum of scalar curvature. 
	
	Furthermore, if the upper sectional curvature is small, then the perturbation is also small in the $C^{1,\alpha}$-sense (cf. Remark \ref{rem-curv}), provided with a uniform scalar curvature lower bound and a diameter upper bound. 
\end{cor}

The above phenomenon can be viewed as a natural complement in an opposite direction of the positive mass theorem.

By the rigidity in \cite{Miao-Wang2016} and Main Theorem, the same conclusion in Corollary \ref{cor-main} holds for a ball of constant curvature $\d$ with radius $\le \frac{\pi}{8\sqrt{\d}}$, after replacing scalar curvature with Ricci curvature. Similar quantitative rigidity should hold when the lower Ricci curvature bound is close to $0$.

In general the mean curvature is too weak to determine the topology of a submanifold. At present we do not know if $M$ in Main Theorem could be very twisted or not. We propose the following conjecture.

Let $q>0$, and let $\|B\|_q=\left(\frac{1}{|M|}\int |B|^q\right)^{1/q}$ is the normalized $L^q$ norm of the 2nd fundamental form $B$ of $M$. 
\begin{conjecture}\label{conj-main}
	If in addition, in Main Theorem $M$ satisfies $|M|^\frac{1}{n}\|B\|_q\leq A$ for some $q>0$, then $M$ is embedded and diffeomorphic to a round sphere.
\end{conjecture}

Our next result verifies Conjecture \ref{conj-main} for $q>n$.

\begin{theo}\label{main-theo-B}
	Let the assumptions be as in Main Theorem. If in addition, 
	\begin{equation}\label{rescaling-inv-bound}
	|M|^\frac{1}{n}\|B\|_q\leq A_2 \quad(\text{ for some $A_2>0$ and $q>n$ }),
	\end{equation}
	then for $0\le \e<\e_1(A_1,A_2, q,R,\d,\mu,n)$,
	\begin{enumerate}\numberwithin{enumi}{theo}
		\item $M$ is embedded, diffeomorphic and $C_2\e^{\min\left\{ \frac{1}{2(2n+1)},\frac{q-n}{2(q-n+qn)}\right\}}$-almost isometric to a geodesic sphere $S(p_0,R_0)$, where $C_2=C_2(n,\d,\mu, q, R,A_1,A_2)$ is a positive constant;
		\item $M$ is  $\varkappa(\epsilon|A_1,A_2, q,R,\d,\mu,n,\alpha)$-$C^{\alpha}$ close to a round sphere of constant curvature $1/(s_\d(R_0))^2$ with $0\le\alpha<1$.
	\end{enumerate}
\end{theo}

The $C^{\alpha}$-closeness of metric tensors in (1.5.2) follows directly from (1.5.1) and  $C^{\alpha}$-regularity \cite[Theorem 2.35]{Tian-Zhang2016} under $L^{q/2}$-integral Ricci curvature bound with $q>n$ and $\varkappa$-non-collapsing condition.

For $0<q\le n$, Conjecture \ref{conj-main} is open even for hypersurfaces in a space form. Note that by Main Theorem, under condition \eqref{rescaling-inv-non-local-col}, pinching phenomena of \eqref{pinching-condition-1} essentially can happen only in space forms, as long as the ambient space around has trivial topology and bounded geometry. We will give some examples in \S \ref{sec:6} that do not satisfy (\ref{rescaling-inv-bound}) but support Conjecture \ref{conj-main}.

Several remarks on Main Theorem and Theorem \ref{main-theo-B} are given in below.
\begin{rem}
	What earlier known about pinching \eqref{pinching-condition-1} is very restrictive when the ambient space is a Riemannian manifold.
	In \cite{Grosjean-Roth2012} Grosjean and Roth proved Theorem \ref{main-theo-B} under some technical assumptions such that the hypersurface $M$ was required to be contained in a small geodesic ball of radius $\le \epsilon$, where $\epsilon$ coincides with the pinching error in (\ref{pinching-condition-1}). Thus in their case $M$ approaches to a point as $\epsilon\to 0$, and (M2) in Main Theorem is trivially satisfied.
	
	According to our proof, the condition of lower bounded sectional curvature, $K_N\ge \mu$, may be weakened (e.g., an integral Ricci curvature lower bound); cf. Remark \ref{rem-fixed-position}.
	
	Under $L^q$ bound $(q>n)$ of 2nd fundamental form \eqref{rescaling-inv-bound}, results corresponding to Theorem \ref{main-theo-B} has been recently studied and known in space forms; see \cite{Colbois-Grosjean2007}, \cite{Aubry-Grosjean} and \cite{Hu-Xu2017}. 
\end{rem}

\begin{rem}
	By the discussion on (\ref{rescaling-inv-non-local-col}) above Main Theorem, typical examples that violate (\ref{rescaling-inv-non-local-col}), and thus are not covered by Main Theorem, contain the boundary $\partial U_r$ of a $r$-neighborhood of a high co-dimensional submanifold $X$ (not a point) with $r \ll\operatorname{diam}X$.
	
	We prove in \cite{Hu-Xu2019} that, the conclusion of Theorem \ref{main-theo-B} holds for
	$\partial U_r$, provided that it is convex. Hence extrinsically collapsed convex hypersurfaces do not satisfy pinching condition (\ref{pinching-condition-1}).
\end{rem}

\begin{rem}\label{rem-curv}
	Under our setting $\mu\le K_N\le \d$, one cannot expect that sectional curvature of  $\Omega_0=B(p_0,R_0)$ is pointwise close to $\d$. It is not difficult to construct a warped product manifold $\Omega$ where $M$ is its slice, and pinching condition (\ref{pinching-condition-1}) holds with arbitrary small $\epsilon$, but there are points in the enclosed domain $\Omega_0$ by $M$ around which the sectional curvature is arbitrarily away from $\d$.
	
	Instead of pointwise curvature closeness, for any $p\ge 1$ $B(p_0,R_0)$ in Main Theorem is almost Einstein in the sense of normalized $L^{p}$-norm, i.e.,
	\begin{equation}
	\|\Ric_{\Omega_0}-n\delta g_{\Omega_0}\|_{p}\le \varkappa(\epsilon|A_1,R,\delta,\mu,n,p),
	\end{equation}
	which directly follows from \cite[Lemma 1.4]{Chen-Rong-Xu2018} (cf. \cite{Anderson1990}) as a standard Schauder estimate.
\end{rem}

\begin{rem}\label{rem-CC-1}
	Unlike the rigidity of (\ref{isoper2}), in our main theorems it is necessary for $M$ to be contained in a ball of radius $\frac{\pi}{2\sqrt{\d}}$ when $\d>0$, though $\frac{\pi}{8\sqrt{\d}}$ (resp. $\frac{\pi}{4\sqrt{\d}}$) in Main Theorem (resp. Theorem \ref{main-theo-A}) is technically required. A counterexample can be easily constructed via smoothing a cylinder $[0,t]\times S^2$ glued with one cap at $0$, where the geometry of boundary at $t$ and interior curvature bound does not change as $t$ varies.

\end{rem} 	

\begin{rem}\label{rem-CC-2}
	Due to natural geometric restrictions, the criteria in Main Theorem generally fails when $B(p,R)$ contains nontrivial topology via cut points. The connected sum of a flat torus $\mathbb T^2$ with a flat disk $D^2$ via a think neck of non-positive curvature glued around center of $D^2$ provides a counterexample.
\end{rem}

\begin{rem}
	Compared with Cheeger-Colding's quantitative rigidity for warped products \cite{Cheeger-Colding1996},
	which overcomes cut points,	
	the ``boundary condition'' via pinching  (\ref{pinching-condition-1}) is less restrictive; i.e.,  Main Theorem is applicable for an (immersed) hypersurface $M$, which could be very twisted  and a priori lie far away from a level set of a warping function. For example, the boundary of a ball may be far away from a level set of a new distance function after an interior metric perturbation, as what has happened in Corollary \ref{cor-main}.
\end{rem}

\begin{rem}\label{rem-almost-rig}
	The principle of almost rigidity behind (M2) is that, if $\mu\le K_N\le \d$ and the Hessian $\nabla^2 r$ of a distance function to a point $p_0\in N$ is close to $\frac{c_\d(R)}{s_\d(R)}g_r$ along the boundary of $B(p_0,R)$ (without knowing whether $|J(R)|$ close to $s_\d(R)$ for a normal Jacobi field $J$), where metric tensor $g_N=dr^2+g_r$, then the interior of $B(p_0,R)$ is almost isometric to that of constant curvature $\d$; see Lemma \ref{lem-quantitative-rigidity}.
	
	In contrast, the principle in Cheeger-Colding's almost rigidity \cite{Cheeger-Colding1996} requires, essentially, the same closeness hold over the whole ball (or more generally, an annulus) in the $L^1$ sense.
\end{rem}

\begin{rem}
	Motivated by Corollary \ref{cor-main} and the quasi-local mass rigidity (cf. \cite{Yau2001}, \cite{Shi-Tam2002}), it is natural to ask that whether the lower scalar curvature bound can be used similarly to detect the interior perturbation of a bounded Euclidean domain? 
	
	Note that by Colding \cite{Colding1997} and Cheeger-Colding \cite{Cheeger-Colding1997}, any interior perturbation of a bounded Euclidean domain cannot be large in the bi-H\"older sense, if it lowers down only a small amount of Ricci curvature. For a quantitative rigidity of positive mass theorem in another direction, see for example \cite{Sormani-Allen2019}.
\end{rem}

In the end of Introduction we point out the main ideas and difficulty in proving Main Theorem.

We will prove the Hausdorff closeness (M1), by first improving estimates in \cite{Grosjean-Roth2012} to conclude that $M$ lies in a small neighborhood of a geodesic sphere $S(p_0,R_0)$, where $p_0$ is the center of mass of $M$ in $B(p,R)$. Then based on an observation in \cite{Colbois-Grosjean2007}, via contacting a ``standard'' sphere-tori to $M$, it is not difficult (see \S \ref{sec:3}) to show that $S(p_0,R_0)$ is also near $M$.

The main ideas in deriving a pointwise estimate on the position of $M$ are from \cite{Colbois-Grosjean2007}; i.e, first to transform pinching condition (\ref{pinching-condition-1}) into an $L^2$ pinching (\ref{ineq-Ie}) on position vector $X$;  then to apply Moser iteration to bound $\|X\|_\infty$. We improve corresponding estimates in \cite{Grosjean-Roth2012} and drop a technical assumption in \cite{Grosjean-Roth2012}, via a careful analysis combined with some geometrical observation in terms of the out-radius $R$. This is done in \S \ref{sec:4}.

By (M1), a naive approach for (M2) is arguing by contradiction. Up to a rescaling, there is a sequence of pairs $(M_i,N_i)$ converging to a limit $(S(p_\infty,1), N_\infty)$ in Gromov-Hausdorff topology, where $N_\infty$ is $C^{1,\alpha}$-Riemannian manifold and $S(p_\infty,1)$ is a geodesic sphere of radius $1$ in $N_\infty$. One may guess the pinching condition (\ref{pinching-condition-1}) can be passed to the limit pair $(S(p_\infty,1), N_\infty))$ with zero pinching error, such that rigidity for the limit may follow from similar arguments as Theorem \ref{main-theo-A}.

According to \cite{Fukaya1987}, by passing to a subsequence, $(M_i,d_i,\operatorname{dvol}_i)$ converges to $(S(p_\infty,1), d_\infty, \mu_\infty))$ in measured Gromov-Hausdorff topology, where $d_i$ and $d_\infty$ are the restricted distance from the ambient space respectively, $\operatorname{dvol}_i$ is the Riemann-Lebesgue measure and $\mu_\infty$ is its limit measure.
By Fukaya's observation in \cite{Fukaya1987}, $\limsup_{i\to \infty}\lambda_1(M_i)\le \lambda_1(S(p_\infty,1),\mu_\infty)$.

If $\mu_\infty$ coincides with Hausdorff measure $\operatorname{dvol}_\infty$ of $S(p_\infty,1)$, then it is easy to apply similar arguments as Theorem \ref{main-theo-A} to derive $B(p_\infty,1)$ is isometric to a ball in space form.

A crucial difficulty is that, if $M_i$ is far away from an embedded diffeomorphic sphere, then $\mu_\infty=\operatorname{dvol}_\infty$ generally fails, and thus the relation between pinching condition (\ref{pinching-condition-1}) and the limit geometry is lost.

Here is our approach. Instead of looking at the limit, we will translate pinching condition (\ref{pinching-condition-1}) along $M$ to its position vector $X$ at $p_0$ such that $X$ is close to $R_0=1$ (up to a rescaling) and perpendicular to $M$ in the $L^2$ sense. By refining the relation between divergence of $X$ on $M$ and mean curvature $H$ (see Lemma \ref{lem-refine-heintze}, cf. Lemma \ref{lem-heintze}), $H$ and $\frac{1}{n}\Delta r$ are close in the $L^1$ sense, where $r=d(p_0,\cdot)$. Since by (M1), $\|H\|_\infty$ is close to $\frac{c_\d}{s_\d}(1)$, we see that $\frac{1}{n}\Delta r$ is also close to $\frac{c_\d}{s_\d}(1)$ along $M$ in the $L^1$ sense. Via the weighted monotonicity of $M$'s volume in extrinsic balls (Lemma \ref{lem-local-ball-volume-estimate}, cf. \cite{Colding-Minicozzi-book}), we transmit the $L^1$ estimate on $\Delta r$ over $M$ to points sufficient dense in $S(p_0,1)$. Then by the almost rigidity principle mentioned in Remark \ref{rem-almost-rig}, we prove that the interior Jacobi fields admit a uniform control, and thus prove (M2). This is done in \S \ref{sec:3}.

The remaining of the paper is organized as follows. We recall some necessary facts and tools in \S \ref{sec:2}, and give a proof of Theorem \ref{main-theo-A} as a preliminary knowledge.
\S \ref{sec:3} and \S \ref{sec:4} are devoted to the proof of Main Theorem. In \S \ref{sec:5} we prove Theorem \ref{main-theo-B}. A series of examples that partially support Conjecture \ref{conj-main} are given in \S \ref{sec:6}. Appendix is for proofs of some technical lemmas.

{\bf Acknowledgements}. The first author was supported by China Postdoctoral Science Foundation (No.2018M641317). The second author was supported partially by National Natural Science Foundation of China [11871349], [11821101], by research funds of Beijing Municipal Education Commission and Youth Innovative Research Team of Capital Normal University.

\section{Preliminaries}\label{sec:2}
In this section we provide notations and facts used later. The $\delta$-sine function $s_\d$ is defined by
\begin{equation}\label{2.1}
\begin{split}
s_\d(r):=\left\{ \begin{aligned}
&\frac{1}{\sqrt{\d}}\sin (\sqrt{\d} r), \quad &\text{if}&~\d>0~\text{and}~r\in \left[0,\frac{\pi}{2\sqrt{\d}}\right);\\
&r,                                     \quad &\text{if}&~\d=0~\text{and}~r\in [0,\infty);\\
&\frac{1}{\sqrt{-\d}}\sinh(\sqrt{-\d}r),\quad &\text{if}&~\d<0~\text{and}~r\in [0,\infty),
\end{aligned}
\right.
\end{split}
\end{equation}
and $\delta$-cosine function is defined by $c_\d(r):=s'_\d(r)$.
Clearly, the following identities hold:
\begin{align*}
c'_\d=-\d s_\d, \quad c^2_\d+\d s_\d^2=1.
\end{align*}

\subsection{Convexity radius}\label{sec:2.1}
Let $N$ be a (maybe non-complete) Riemannian manifold. The exponential map, $\exp_p:T_pN\to N$, from tangent space at $p$ to $N$ is well-defined locally. The {\em injectivity radius} of a point $p\in N$, $\inj(p)$, is defined to be the supremum of radii of open balls $B(o,r)$ centered at origin $o$ of $T_pN$ where the restriction $\exp_p|_{B(o,r)}$ is a well-defined diffeomorphism onto its image. The {\em conjugate radius} of $p$, $\conj(p)$, is defined to be the supremum of radii of open balls $B(o,r)$ in $T_pN$ which contains no critical point of $\exp_p$. The {\em convexity radius} of $p$, $\conv(p)$, is defined to be the supremum of radii of open balls centered at $p$ that is strongly convex (cf. \cite{Xu2018}). We call an open set $U\subset N$ {\em convex}, if any two points of $U$ are joined by a unique minimal geodesic in $N$ and its image lies in $U$. An open ball $B(p,r)$ is called {\em strongly convex}, if any $B(q,s)\subset B(p,r)$ is convex.

By definition, $\conv(p)\le \inj(p)$, and  $\inj(p)\le \conj(p)$. The following pointwise estimates of $\inj(p)$ and $\conv(p)$ will be used later.

\begin{lem}[\cite{Xu2018},\cite{Mei2016}]\label{theo-2.2}
	Assume that $B(p,2R)$ has a compact closure in $N$. For any point $q\in B(p,R)$, the followings hold.
	\begin{align}\label{2.3}
	\inj(q)&\geq \min\{\inj(p),\conj(q)\}-d(p,q),\\
	\label{2.4}
	\conv(p)&\geq \frac{1}{2}\min\{\frac{\pi}{\sqrt{\d}},\inj(p)\},
	\end{align}
	where $\d=\sup\{K_N(x): x\in B(p,\inj(p))\}$ is the upper bound of sectional curvature in $B(p,\inj(p)).$
\end{lem}
Lemma \ref{theo-2.2} was first proved by Mei \cite{Mei2016}, where $\conj(q)$ was replaced by its lower bound $\frac{\pi}{\sqrt{\d}}$. Later (\ref{2.3}) and a curvature-free version of (\ref{2.4}) was proved by the second author, where $\frac{\pi}{\sqrt{\d}}$ is replaced by the focal radius; see \cite{Xu2018}.

\subsection{Sobolev inequality}\label{sec:2.2}
The well-known Sobolev inequality for Riemannian submanifolds due to Hoffman and Spruck \cite{Hoffman-spruck1974} is a fundamental tool applied in the proof of Main Theorem.
\begin{theo}[\cite{Hoffman-spruck1974}]\label{theo-2.1}
	Let $N$ be a complete Riemannian manifold with $K_{N}\leq \d$. Let $M$ be a compact immersed submanifold in $N$. Let $f\in C^1(M)$ be nonnegative. For $\d>0$, in addition we assume that the volume of $M$ has an upper bound,
	\begin{equation}\label{vol-inj-require}
		|M|<\omega_n s_\d^n\(\frac{\min_{p\in M}\inj(p)}{2}\).
	\end{equation}
	Then there exists a positive constant $C=C(n)$ such that
	\begin{align}\label{Hoffman-Spruck}
	\(\int_{M} f^\frac{n}{n-1} \)^\frac{n-1}{n}\leq C\int_{M}(|\nabla f|+f|H|).
	\end{align}
\end{theo}

We now verify that Theorem \ref{theo-2.1} is applicable for hypersurface $M$ in Main Theorem by
localizing injectivity radius along whole $M$ to one point.

Let $M^n$ be a compact hypersurface immersed into a geodesic contractible ball $B(p,R)$, where sectional curvature of ambient space $K_N\le \d$ and $R\le\frac{\pi}{2\sqrt{\d}}$.
In order to apply Theorem \ref{theo-2.1} for $\d>0$, we need to justify (\ref{vol-inj-require}) under the condition of Main Theorem.

Let us consider the ball $B(o,\frac{\pi}{\sqrt{\d}})$ in the tangent space $T_pN$ with the pullback metric by $\exp_p$. Since $\inj(o)>\frac{\pi}{\sqrt{\d}}-\epsilon$ for any $\epsilon>0$, it is easy to see by Lemma \ref{theo-2.2} that, for any $q\in B(o,\frac{\pi}{2\sqrt{\d}})$, $\inj(q)\ge \frac{\pi}{2\sqrt{\d}}$.
	
On the other hand, since $B(p,R)$ is geodesic contractible, $M$ can be lifted by the inverse of $\exp_p|_{B(p,\inj(p))}$ into $B(o,\frac{\pi}{\sqrt{\d}})$. Then by \eqref{2.3}, for any point $q$ of the image of $M$ in $B(o,\frac{\pi}{\sqrt{\d}})$, $\inj(q)> \frac{\pi}{2\sqrt{\d}}$.

Therefore, if $M$ lies in a geodesic contractible ball $B(p,R)$, then (\ref{vol-inj-require}) can be replaced by
\begin{equation}
R\le\frac{\pi}{2\sqrt{\d}}, \qquad |M|\le \omega_n s_\d^n(\frac{\pi}{4\sqrt{\d}}).
\end{equation}
So (\ref{location-size}) implies (\ref{vol-inj-require}).

\subsection{Convergence theorems in Gromov-Hausdorff topology}\label{sec:2.3}
We recall convergence results for Riemannian manifolds under Gromov-Hausdorff topology.

We say that a sequence $(X_i,d_i)$ of metric spaces {\em $GH$-converges} to $(X,d)$ in Gromov-Hausdorff topology, denoted by $(X_i,d_i)\overset{\operatorname{GH}}{\longrightarrow}(X,d)$, if there is $\epsilon_i$-isometries $\psi_i:X_i\to X$ with $\epsilon_i\to 0$, i.e., for any $x,y\in X_i$, $|d_i(x,y)-d(\psi_i(x),\psi_i(y))|\le\epsilon_i$, and $\epsilon_i$-neighborhood of $\psi_i(X_i)$ covers $X$.

Gromov's compactness theorem (cf. \cite{Burago-Burago-Ivanov2001}) says that for any collection $\{(X_\alpha,d_\alpha)\}$ of compact metric spaces of bounded diameter, if they are uniformly and totally bounded (i.e., there is a nonnegative function $\tau$ such that for each $\alpha$, any maximal $\epsilon$-discrete net of $X_\alpha$ contains points at most $\tau(\epsilon)$), then $\{(X_\alpha,d_\alpha)\}$ is precompact (i.e,. has a compact closure) in the Gromov-Hausdorff topology. Such precompactness can be guaranteed by the relative volume comparison theorem under lower bounded Ricci curvature.

For $n$-manifolds with uniformly bounded (sectional or Ricci) curvature and under certain appropriate non-collapsing assumptions, $GH$-convergence implies higher regularity of metric tensors.

\begin{theo}[Cheeger-Gromov's convergence, \cite{Cheegerphd,Cheeger1970}, \cite{GLP1981}, \cite{Kasue1989}, \cite{Green-Wu1988}, \cite{Peters1987}]\label{thm-ChGro-Convergence}
	Let $(M_i,g_i)$ be a sequence of Riemannian $n$-manifolds whose sectional curvature $|K_{(M_i,g_i)}|\le 1$, diameter $\le d$ and injectivity radius $\inj(M,g_i)\ge \rho$. Then there is a subsequence $(M_{i_1},g_{i_1})$ whose $\operatorname{GH}$-limit is isometric to a $C^{1,\alpha}$-Riemannian manifold $(M,g)$, and there are diffeomorphisms $f_{i_1}:M\to M_{i_1}$ for all sufficient large $i_1$ such that the pullback metric $f_{i_1}^*g_{i_1}$ converges to $g$ in the $C^{1,\alpha}$ topology for any $\alpha\in [0,1)$ (i.e., there is a fixed coordinates system such that $f_{i_1}^*g_{i_1}$ converges to $g$ on each chart in the $C^{1,\alpha}$-norm).
\end{theo}

A compact Riemannian $n$-manifold $(M,g)$ is said to be of {\em $(H,L^p)$-bounded Ricci curvature}, if for real numbers $p>n/2$ and $H\ge 1$,
$$\int_M |\operatorname{Ric}|^p\le H.$$
A manifold $(M,g)$ is called to be  \emph{$\varkappa$-non-collapsing with $\varkappa=1-\eta$ at scale $r_0$}, if
$$\operatorname{vol}(B(x, r))\ge (1-\eta)\operatorname{vol}(B(o,r))\qquad  \text{for all $x\in M$ and $r \le r_0$},$$
where $B(o,r)$ is an Euclidean ball of radius $r$.

By \cite{Petersen-Wei1997}, the relative volume comparison of balls holds on manifolds of $L^p$-bounded Ricci curvature. Since in harmonic coordinates, the $L^p$-bound of the Ricci curvature gives the $W^{2,P}$-bound of the metric tensor $g_{ij}$ (cf. \cite{Anderson1990}), by the $C^{\alpha}$-harmonic radius estimate \cite[Theorem 2.35]{Tian-Zhang2016},  the following $C^{\alpha}$-regularity convergence result holds.
\begin{theo}[\cite{Tian-Zhang2016},\cite{Petersen-Wei1997}]\label{thm-Calpha-convergence}
	For any $0<\alpha<1$, there is $\eta>0$ such that for any sequence $(M_i,g_i)$ of Riemannian $n$-manifolds of
	$(H,L^p)$-bounded Ricci curvature, $(1-\eta)$-non-collapsing volume at scale $r_0$ and uniformly bounded diameter, there is a subsequence $(M_{i_1},g_{i_1})$ whose $GH$-limit is isometric to a $C^{\alpha}$-Riemannian manifold $(M,g)$, and there are diffeomorphisms $f_{i_1}:M\to M_{i_1}$ for all sufficient large $i_1$ such that the pullback metric $f_{i_1}^*g_{i_1}$ converges to $g$ in the $C^{\alpha}$ topology.
\end{theo}

\subsection{Center of mass}\label{sec:2.4}
Let $M$ be an immersed submanifold in a geodesic contractible ball $B(p,R)\subset N$. If in addition $K_N\le \d$ with $\d>0$, then we assume $R\le \frac{\pi}{4\sqrt{\d}}$. By lifting $M$ to $T_pM$ as the same argument below Theorem \ref{theo-2.1} and by \eqref{2.4}, we assume without loss of generality that $B(p,2R)$ is convex.

Let $\mathcal{F}: B(p,2R)\ra \R$ be an energy function defined by
\begin{align*}
\mathcal{F}(q):=\int_{M} \Phi_\d(\mrm{dist}(q,x)) dx,
\end{align*}
where $\Phi_\d$ is the modified distance function  defined by
\begin{equation}\label{2.2}
\Phi_\d(r):=\int_0^rs_\d(s)ds.
\end{equation}
We claim that there is a unique minimum point $p_0\in B(p,R)$ of $\mathcal{F}$ in $B(p,2R)$. We call $p_0\in N$ the {\em center of mass} of $M$ with respect to modified distance.

First, by $R\le \frac{\pi}{4\sqrt{\d}}$ and \eqref{2.4}, every $\Phi_\d(\operatorname{dist}(x,\cdot))$ is convex. Hence $\mathcal{F}$ is a strictly convex function on $B(p,2R)$, which admits a unique minimum point $p_0\in B(p,2R)$ such that $\nabla^N \mathcal F(p_0)=0$. Note that it is equivalent to $Y(p_0)=0$, where $Y$ is a vector field defined by
\begin{align}\label{2.6}
Y(q):=\int_M \frac{s_\d(r)}{r}\exp_q^{-1}(x)\in T_qN,
\end{align}
where $r(x)=\mrm{dist}(x,q)$.

Secondly, because the vector field $Y$ defined above, by the convexity of $B(p,R)$, points into interior of $B(p,R)$ along the boundary. It follows that the minimum point $p_0$ of $\mathcal{F}$ lies in $B(p,R)$.

By definition, it is clear that in the normal coordinates $\{x_1,\cdots,x_{n+1}\}$ of $p_0$, \eqref{2.6} becomes
\begin{align}\label{center-mass-test-func}
\int_M \frac{s_\d(r)}{r}x_i=0, \qquad i=1,\dots,n+1,
\end{align}
where $r(x)=\mrm{dist}(x,p_0)$.

By the discussion above, the hypersurface $M$ in Main Theorem, Theorems \ref{main-theo-A} and \ref{main-theo-B} always admits a center of mass in $B(p,R)$.

\subsection{Test functions for 1st eigenvalue of Laplace-Beltrami operator}
Now let $M\subset B(p,R)\subset N$ be an immersed oriented hypersurface with $K_N\le \d$. If $\d>0$, we further assume that $R\le \frac{\pi}{4\sqrt{\d}}$.

Let $p_0$ be the center of mass of $M$ and $r(x)=d(x,p_0)$. We call the vector field $X=s_\d(r)\nabla^N r$ the {\em position vector} about $p_0$. By Rayleigh's principle and (\ref{center-mass-test-func}), each component of the position vector $X$ provides a test function of $\lambda_1$, such that
\begin{align}
\l_1 \int_M |X|^2=\l_1 \int_M \sum_{i=1}^{n+1}\left(\frac{s_\d(r)}{r} x_i\right)^2 \leq & \int_M \sum_{i=1}^{n+1}\left|\nabla^M \frac{s_\d(r)}{r} x_i\right|^2
\end{align}

In order to estimate $\nabla^M \frac{s_\d(r)}{r} x_i$, we need the following lemmas from \cite{Heintze1988}. Let $X^\top$ be $X$'s tangential projection over $M$, then
$$
X^\top_x=s_\d(r(x))\nabla^M r|_x, \quad x\in M.
$$
For any vector field $Y$ on $N$, the {\em divergence} of $Y$ along $M$ is defined by
$$
\divv_M Y(p):=\sum_{i=1}^{n}\langle \nabla^N_{e_i}Y,e_i\rangle,
$$
where $\{e_1,\cdots,e_n \}$ is an orthonormal basis of $T_p M$.

\begin{lem}[\cite{Heintze1988}]\label{lem-heintze}
	Let the assumptions be as above. The following inequalities hold.
	
	(i) Let $\nu$ be the normal vector of $M$ from its orientation, then
	\begin{align}\label{3.1}
	\divv_M X^\top\geq nc_\d-nH\langle X,\nu \rangle,
	\end{align}

	(ii) The covariant derivatives of components of $X$, $s_\d (r)x_i/r$ satisfies
	\begin{align}\label{3.2}
	\sum_{i=1}^{n+1}\left|\nabla^M \(\frac{s_\d(r)}{r}x_i\) \right|^2+ \d|X^\top|^2 \leq n,
	\end{align}
	
\end{lem}
\begin{proof}
	Since the proof of (i) is used in the proof of Main Theorem, we give a proof. For (ii) we refer to \cite{Heintze1988}.
	
	Let $p\in M$ and $\{e_1,\cdots,e_n\}$ be an orthonormal basis of $T_p M$. If $\nabla^M r=0$ at $p$, then $e_j\bot \nabla^N r$ for all $j\in\{1,\cdots,n\}$; Otherwise, we take $e_n=\frac{\nabla^M r}{|\nabla^M r|}$ and $e_j\bot \nabla^N r$ for $j\in \{1,\cdots,n-1\}$. Then we get
	$$
	e_n= |\nabla^M r|\nabla^N r+\sqrt{1-|\nabla^M r|^2}e_n^\ast,
	$$
	where $e_n^\ast$ is a unit vector such that $e_n^\ast \bot \nabla^N r$. By Hessian comparison theorem for $K_N\leq \d$, we get
	\begin{equation*}
	\begin{split}
	\divv_M \nabla^N r(p)=& \sum_{j=1}^{n}\langle \nabla^N_{e_j} \nabla^N r,e_j \rangle\\
	=&\sum_{j=1}^{n-1}\langle \nabla^N_{e_j} \nabla^N r,e_j \rangle+(1-|\nabla^M r|^2)\langle \nabla^N_{e_n^\ast} \nabla^N r,e_n^\ast \rangle \\
	\geq & (n-|\nabla^M r|^2)\frac{c_\d}{s_\d},
	\end{split}
	\end{equation*}
	where we used the fact that $\Hess^N r(\nabla^N r,\cdot)=0$ and the standard Jacobi field estimates. Thus we have
	\begin{align*}
	\divv_M X=s_\d \divv_M \nabla^N r+c_\d |\nabla^M r|^2\geq n c_\d.
	\end{align*}
	By the identity
	\begin{align*}
	\divv_M X^\bot=\sum_{i=1}^n \langle \nabla^N_{e_i} X^\bot,e_i \rangle=\sum_{i=1}^{n} e_i\langle X^\bot,e_i \rangle-\sum_{i=1}^n\langle X^\bot,\nabla^N_{e_i}e_i\rangle=\langle X^\bot,nH\nu\rangle=nH\langle X,\nu\rangle,
	\end{align*}
	we have
	$$
	\divv_M X^\top= \divv_M X-\divv_M X^\bot=\divv_M X-n\langle X,H\nu\rangle\geq n c_\d-nH\langle X,\nu\rangle.
	$$
\end{proof}

\subsection{Proof of Theorem \ref{main-theo-A}}

Theorem \ref{main-theo-A} is an observation based on Heintze \cite{Heintze1988}. As one of the preliminaries, we give a direct proof. The following inequalities are used.
\begin{lem}
	Integral of (\ref{3.1}) gives
	\begin{align}\label{3.4}
	\int_M c_\d \leq \int_M |H| s_\d.
	\end{align}
	Let $c=\frac{1}{|M|}\int_M c_\d$. Then by (\ref{3.4}) and Cauchy-Schwarz inequality
	\begin{equation}\label{thma-testfunc-2}
	(1-c^2)\(n\d+\frac{n}{|M|}\int_M |H|^2\)\ge n\delta,
	\end{equation}
	By (\ref{3.1}) and integrating $\divv_M (c_\d X^\top)=-\delta s_\d\langle  \nabla^M r,X^\top\rangle + c_\d \divv_M X^\top$, it gives
	\begin{align}\label{3.5}
	\d \int_M |X^\top|^2 \geq n\int_M c_\d^2-n\int_M |H| s_\d c_\d.
	\end{align}
	Furthermore, if $\d<0$ then by identity $\d s_\d^2+c_\d^2=1$ and Cauchy-Schwarz inequality,
	\begin{align}\label{3.6}
	\int_M s_\d \int_M s_\d c_\d \leq \int_M s_\d^2 \int_M c_\d.
	\end{align}
\end{lem}
\begin{proof}
	(\ref{3.4}) and (\ref{3.5}) are by direct calculations. We refer to \cite[Lemma 2.8]{Heintze1988} for a proof of (\ref{3.6}). The verification of (\ref{thma-testfunc-2}) can be done by direct calculation as follows.
	Since $c_\d\ge 0$ along $M\subset B(p_0,2R)$ with $R\le\frac{\pi}{4\sqrt{\delta}}$,
	\begin{align*}
	&(1-c^2)\(\d+\frac{1}{|M|}\int_M |H|^2\) \\
	\overset{(\ref{3.4})}{\geq} & \d-\d \(\frac{1}{|M|}\int_M |H| s_\d\)^2 +\frac{1}{|M|}\int_M |H|^2-\frac{1}{|M|^2}\int_M |H|^2 \int_M c_\d^2 \\
	\geq  \;\;& \d-\frac{1}{|M|^2}\int_M |H|^2\int_M ( \d s_\d^2-1+c_\d^2)=\d.
	\end{align*}
\end{proof}

\begin{proof}[Proof of Theorem \ref{main-theo-A}]
	~
	
	Let $p_0$ be the center of mass of $M$ and $r(x)=d(x,p_0)$, and let $X=s_\d(r)\nabla^N r$ be the position vector with respect to $p_0$ and $X^\top$ be its tangential projection over $M$.
	We claim that
	
	\begin{claim}
		If (\ref{1steigeneq-1}) (resp. (\ref{1steigeneq-2})) holds for $\d\ge 0$ (resp. $\d<0$), then  $\nabla^M r=0$ and $|H|=\frac{c_\d}{s_\d}(r)$.
	\end{claim}
	
	By the claim, the image of $M$ is a geodesic sphere centered at $p_0$, and the Laplacian of $\Delta r=|H|=\frac{c_\d}{s_\d}$. Since $K_N\le \d$, it implies
	\begin{align*}
	\mrm{Hess}^N r=\frac{c_\d(r)}{s_\d(r)}g_r, \quad \text{on $\Omega$}.
	\end{align*}
	By the rigidity of Hessian comparison, $\Omega$ must be a geodesic ball of constant curvature $\d$.
	
	If $n\ge2$, then by the simply connectedness of $S^n=\partial \Omega$, $M$ is embedded. For $n=1$, by the fact that $\lambda_1(M)=\frac{4\pi^2}{L^2(M)}$, where $L(M)$ is $M$'s length, and the 1st eigenvalue of its image $\partial \Omega$ satisfies (\ref{1steigeneq-1}) or (\ref{1steigeneq-2}), it is clear that $L(M)=L(\partial \Omega)$. Hence $M$ is also embedded.
	
	The claim can be directly seen from Heintze's proof \cite{Heintze1988}. For completeness, we give its verification below by dividing into three cases: $\d=0$, $\d>0$ and $\d<0$.
	
	Case 1 for $\d=0$. Let us take $x_i$, $i=1,\cdots,n+1$ as test functions in Rayleigh quotient. Since $p_0$ is the center of mass of $M$, we have $\int_{M}x_i=0$. By Rayleigh quotient,
	\begin{align}\label{thma-raylei-1}
	\l_1 \int_M |X|^2=\l_1 \int_M \sum_{i=1}^{n+1}x_i^2 \leq & \int_M \sum_{i=1}^{n+1}\left|\nabla^M x_i\right|^2
	\end{align}
	Since by (\ref{1steigeneq-1}), the LHS of (\ref{thma-raylei-1}) equals to $\frac{n}{|M|}\int_M |H|^2 \int_M |X|^2$. At the same time, by (\ref{3.2}) and (\ref{3.4}), the RHS satisfies
	\begin{align*}
	\int_M \sum_{i=1}^{n+1}\left|\nabla^M x_i\right|^2
	\overset{(\ref{3.2})}\leq  n|M| \overset{(\ref{3.4})}\leq \frac{n}{|M|}\(\int_M |H||X|\)^2
	\leq \frac{n}{|M|}\int_M |H|^2 \int_M |X|^2.
	\end{align*}
	Thus, all inequalities above becomes equality. In particular, by equality in (\ref{3.4}) and integrating (\ref{3.1}), we derive
	\begin{align}\label{thma-claim-ineq}
	0=\int_{M}\divv_M X^\top \geq n\int_{\Sigma} c_\d - n \int_M H\langle X,\nu \rangle\geq n \int_M c_\d-n\int_M |H|s_\d =0.
	\end{align}
	Therefore, $|\langle X,\nu \rangle|=s_\d$, $|H|=\Delta r\le \frac{c_\d}{s_\d}$. By (\ref{thma-claim-ineq}) again $|H|=\frac{c_\d}{s_\d}(r)$.
	
	Case 2 for $\d>0$. Let us take $\frac{s_\d(r)}{r} x_i$, $i=1,\cdots,n+1$ and $\frac{c_\d(r)-c}{\sqrt{\d}}$ as test functions, where $c=\frac{1}{|M|}\int_M c_\d$.
	Since $s_\d^2=|X|^2$ and $\nabla^M \(\frac{c_\d(r)-c}{\sqrt{\d}}\)=\sqrt{\d}X^\top$, we derive
	\begin{align}\label{thma-raylei-2}
	\l_1 \int_M \left[ s_\d^2+\frac{(c_\d-c)^2}{\d} \right] \leq \int_M \left[ \sum_{i=1}^{n+1}\left|\nabla^M \(\frac{s_\d(r)}{r}x_i\)\right|^2+ \d |X^\top|^2\right] \overset{(\ref{3.2})}\leq n|M|,
	\end{align}
	By definition of $c$ and direct calculation, the LHS of (\ref{thma-raylei-2}) equals to $\lambda_1(1-c^2)|M|/\d$.
	Hence we get $$\l_1(1-c^2)\leq n\d.$$
	Combining with (\ref{thma-testfunc-2}), we see that if (\ref{1steigeneq-1}) holds then equality in (\ref{thma-testfunc-2}) must also hold. By the proof of (\ref{thma-testfunc-2}), it in turn implies (\ref{3.4}) is an equality. Now by (\ref{thma-claim-ineq}) again, $\nabla^M r=0$ and $|H|=\frac{c_\d}{s_\d}$.
	
	Case 3 for $\d<0$. As the same for case 1, we take $\frac{s_\d(r)}{r} x_i$, $i=1,\cdots,n+1$ as test functions. Then by (\ref{1steigeneq-2}),
	\begin{align}\label{thma-raylei-3}
	n\d\int_M s_\d^2+n \cdot \max_M |H|^2  \int_M s_\d^2=\l_1 \int_M s_\d^2\leq &\int_M \sum_{i=1}^{n+1}\left|\nabla^M \(\frac{s_\d(r)}{r}x_i\)\right|^2.
	\end{align}
	At the same time we have
	\begin{align}\label{thma-cal-3}
	\text{RHS of (\ref{thma-raylei-3})}  &\overset{(\ref{3.2})}{\leq} \int_M \(n-\d|X^\top|^2\) \nonumber \\
	&\overset{(\ref{3.5})}{\leq} \int_M \(n-nc_\d^2+n |H|s_\d c_\d\)\nonumber \\
	& \leq n\d\int_M s_\d^2+n \cdot \max_M |H| \cdot \int_M s_\d c_\d \nonumber  \\
	& \overset{(\ref{3.6})}{\leq} n\d\int_M s_\d^2+n \cdot \max_M |H| \cdot \int_M s_\d^2 \cdot \frac{\int_M c_\d}{\int_M s_\d} \nonumber \\
	& \overset{(\ref{3.4})}{\leq} n\d\int_M s_\d^2+n \cdot \max_M |H|^2  \int_M s_\d^2.
	\end{align}
	Combining with (\ref{thma-raylei-3}) and (\ref{thma-cal-3}), we see that equality in (\ref{3.4}) holds. By considering (\ref{thma-claim-ineq}) again, the same argument as for case 1 implies the claim.
	
\end{proof}

\begin{rem}
	It should be pointed out that in Theorem \ref{main-theo-A} for $n=1$, Reilly (and also Heintze \cite{Heintze1988}) stated only that $M$ is immersed onto a circle; see \cite[Theorem A]{Reilly1977}. We observe that $M$ is also embedded.
	
	Indeed, for an immersed closed curve $\g:[0,L]\ra \R^2$ of length $L$, $\l_1(\g)=\frac{4\pi^2}{L^2}$. If $\l_1(\g)=\frac{1}{L}\int_\g \k_g^2$, where $\k_g$ is the geodesic curvature of $\g$, then by \cite[Theorem A]{Reilly1977} $\g$ is immersed in some circle of $\R^2$. By rescaling if necessary, we may assume that the radius of the circle is $1$, then $\k_g\equiv 1$ on $\g$ and $L=2\pi$. Thus, $\g$ coincides with the unit circle in $\R^2$.
	
	Main Theorem should hold for $n=1$, i.e., curves in a Riemannian surface, which will be discussed elsewhere.
\end{rem}

\section{Proof of Main Theorem}\label{sec:3}

The proof of Main Theorem is divided into two parts.

Let $p_0$ be the center of mass of the immersed hypersurface $M$ in $N$, and let $R_0=s_\d^{-1}(1/\sqrt{\d+\|H\|_\infty^2})$. Note that by Heintze-Reilly's inequality (\ref{1steigen-ineq}), $\d+\|H\|_\infty^2$ is always positive.

In Part I, we prove the position vector has norm close to $s_\d(R_0)$.

\begin{theo}\label{main-theo-1}
	Under the assumptions of Main Theorem, there is a positive $\epsilon_1=\epsilon_1(A_1, R, \delta, n)$ such that  pinching condition (\ref{pinching-condition-1}) with $0\le \epsilon<\epsilon_1$ implies that
	\begin{equation}\label{main-theo-1.1}
	\left| |X|_x-\frac{1}{\sqrt{\d+\|H\|_\infty^2}}  \right| \leq \frac{C_1\e^\frac{1}{2(2n+1)}}{\sqrt{\d+\|H\|_\infty^2}}, \quad \text{for all $x\in M$},
	\end{equation}
	where $C_1=C_1(n,\d,R,A_1)$ is a constant.
\end{theo}

We point it out that Theorem  \ref{main-theo-1} still holds when $K_N$ only admits an upper bound $\d$; see Remark \ref{rem-fixed-position}.

The proof of Theorem \ref{main-theo-1} is left to next section.

Part II. We prove (M1) and (M2), i.e., $M$ is Hausdorff close to geodesic sphere $S(p_0,R_0)$ and $B(p_0,R_0)$ is $C^{1,\alpha}$-close to a ball of constant curvature.

We need the following bound on the mean curvature $H$ of $M$.

\begin{lem}\label{lem-bound-mean-curv}
	Let $M^n$ be an immersed, oriented, connected and closed hypersurface in a geodesic contractible  ball $B(p,R)$ of $N^{n+1}$, where $K_N \leq \d$. If $\d>0$, we further assume that $R\leq \frac{\pi}{8\sqrt{\d}}$. Then	
	\begin{align}\label{lower-bound-mean-curv}
	\|H\|_\infty &\geq \frac{c_\d}{s_\d}(2R)\ge \sqrt{|\d|},\\
	\label{ineq-2side-bound-ratio-mean-curv}
	\min\left\{1,c_\d(2R)\right\}& \leq \frac{\|H\|_\infty}{\sqrt{\d+\|H\|_\infty^2}} \leq  \max\left\{1,c_\d(2R)\right\}.
	\end{align}
\end{lem}
\begin{proof}
	For (\ref{lower-bound-mean-curv}), let us consider the supporting sphere outside of $M$. It is easy to see (c.f. \cite[Theorem 1]{Jorge-Xavier1981}) that the mean curvature of $M$ is no less than that of the supporting sphere, which by Jacobi comparison $\ge \frac{c_\d}{s_\d}(2R)$.
	
	For (\ref{ineq-2side-bound-ratio-mean-curv}), let us define a function $f(x):=\frac{x}{\sqrt{\d+x^2}}$. Then $f'(x)=\frac{\d}{(\sqrt{\d+x^2})^3}$. By (\ref{lower-bound-mean-curv}), we have $\|H\|_\infty\geq \frac{c_\d}{s_\d}(2R)$. Since $f(x)$ is monotone increasing when $\d\geq 0$, we get $c_\d(2R)\leq \frac{\|H\|_\infty}{\sqrt{\d+\|H\|_\infty^2}}\leq 1$. If $\d<0$, then $f(x)$ is monotone decreasing and $1\leq \frac{\|H\|_\infty}{\sqrt{\d+\|H\|_\infty^2}}\leq c_\d(2R)$.
\end{proof}

Let us prove (M1) first.

\begin{proof}[Proof of (M1) in Main Theorem]
	~
	
	Let us lift $M$ to $T_{p_0}N$, and identify the lift as $M$ itself. Since $\mu \le K_N\le \d$, the pull-back metric $g^*=\exp_{p_0}^*g$ is equivalent to $g_{p_0}$.
	
	By the definition of Hausdorff distance, we will prove that
	$S(0,R_0)$ lies in
	$4(2n-1)\e_1\cdot s_\d(R_0)$-neighborhood of $M$ with respect to the Euclidean metric $(T_{p_0}N,g_{p_0})$, where $\e_1=C_1\e^\frac{1}{2(2n+1)}$ and $C_1$ is the constant in Theorem \ref{main-theo-1}.
	
	Let us argue by contradiction. If it fails for some $\e$, then there is a vector $v \in T_{p_0}N$ such that $|v|=R_0$ and $B(v,\eta)\cap M=\emptyset$, where $\eta=4(2n-1)\e_1\cdot s_\d(R_0)$.

	 Let $T_\eta(t)$ be a family of sphere-tori with respect to $(T_{p_0}N,g_{p_0})$ given by the embedding
	\begin{align*}
	T_\eta(t): S^{n-1}\times S^{1} \to & \mathbb R^{n+1}=\mathbb R^n\oplus \mathbb R v,\\
                 (\xi,\theta)\mapsto &  (R_0 \sin t-2\eta \cos \theta) \xi + (R_0\cos t+2\eta \sin \theta)v,
	\end{align*}
	such that $S^1$ is a small circle of radius $2\eta$, which rotates around $\mathbb R v$-axis at the level $R_0\cos t$ with $t_0< t\le \frac{\pi}{2}$, where $t_0>0$ is determined by $R_0\sin t_0=2\eta$.
	
	Since the small circle contributes a large principal curvature, by a direct calculation the mean curvature $H_{T_\eta(t),g_{p_0}}$ with respect to the Euclidean metric $g_{p_0}$ satisfies
	$$|H_{T_\eta(t),g_{p_0}}(y)|\ge \frac{1}{4n\eta}, \qquad \text{for any $y\in T_\eta(t)$}.$$
	
	Let $H_{T_\eta(t),g^*}$ be the mean curvature of $T_\eta(t)$ with respect to the pull-back metric $g^*$.
	We claim that
	\begin{claim}\label{claim-hausclose}
		As $\eta\to 0$,
		$$\frac{H_{T_\eta(t),g^*}}{H_{T_\eta(t),g_{p_0}}}\to 1.$$
	\end{claim}

	Assuming the Claim, we continue the proof of (M1).
	
	As observed by Colbois-Grosjean \cite{Colbois-Grosjean2007}, by shifting $T_\eta(t)$ downwards along $tv$, we are able to find a point $w\in T_\eta(t_1)$ contact to $M$ for some $t_0< t_1< \frac{\pi}{2}$, which implies that the mean curvature of $H_M(w)$ is not less than that of $H_{T_\eta(t_1),g^*}(w)$.
	
	By Claim \ref{claim-hausclose}, at the contact point $w$, we have
	$$
	|H_M|(w)\ge \frac{1}{2n\eta}=\frac{1}{8n(2n-1)\epsilon_1\cdot s_\d(R_0)}.
	$$
	Since $\epsilon$ can be chosen to be sufficienly small such that 
	$$
	|H_M|(w)\ge 2\frac{c_\d}{s_\d}(R_0)>\|H_M\|_\infty,
	$$ 
	which is a contradiction.
	
	What remains is to verify Claim \ref{claim-hausclose}.
	
	Arguing by contradiction. If there are points $x_\eta$ in $T_\eta(t_\eta)$ such that  $\frac{H_{T_\eta(t),g^*}}{H_{T_\eta(t),g_{p_0}}}$ is definitely away from $1$ at $x_\eta$ as $\eta\to 0$. Let us consider the rescaled metrics $\frac{1}{4\eta^2}g_{p_0}$ and $\frac{1}{4\eta^2}g^*$. Then by identifying $\frac{1}{4\eta^2}g_{p_0}$ as one Euclidean space $(\mathbb R^{n+1},g_0)$, $x_\eta$ being the same point, and $v$ lying on the same line, the sphere-tori $T_\eta(t)$ can be written as the following parametrization
	$$ (\xi,\theta)\mapsto (\frac{R_0}{2\eta} \sin t-\cos \theta) \xi + (\frac{R_0}{2\eta}\cos t+\sin \theta)v.$$
	
	When $\eta\to 0$, by definition $\epsilon_1\to 0$ and $R_0/\eta\to \infty$. Hence, by passing to a subsequence $T_\eta(t_\eta)$ with marked point $x_\eta$, it $C^\infty$-converges to a unit cylinder $\mathbb R^{n-1}\times S^1$ or some sphere-torus isometric to
	$T_{1;s_\infty}: S^{n-1}\times S^1\to \mathbb R^{n+1}$ ($1<s_\infty<\infty$, $t_\infty=\lim t_\eta$),
	$$(\xi,\theta)\mapsto (s_\infty \sin t_\infty-\cos \theta) \xi + (s_\infty\cos t_\infty+\sin \theta)v,$$
	 in $(\mathbb R^{n+1},g_0)$.
	
	Furthermore, the metric $\frac{1}{4\eta^2}g^*$ expressed in $(\mathbb R^{n+1},g_0)$ $C^{1,\alpha}$-converges to $g_0$ as $\eta \to 0$.
	Thus both $H_{T_\eta(t_\eta),\frac{1}{4\eta^2}g^*}$ and
	$H_{T_\eta(t_\eta),g_{0}}$ converge to that of the limit cylinder or sphere-torus.
	Since the quotient $\frac{H_{T_\eta(t),g^*}}{H_{T_\eta(t),g_{p_0}}}$ is rescaling invariant, a contradiction is derived.
	
\end{proof}

Next, we prove (M2) in Main Theorem.
By (M1), $M$ is Hausdorff close to the metric sphere $B(p_0,R_0)$. Moreover, by Lemma \ref{lem-bound-mean-curv}, $\sqrt{\d+\|H\|_\infty^2}$ admits a universal lower bound. Up to a rescaling, we may assume that $R_0=s_\d^{-1}(1/\sqrt{\d+\|H\|_\infty^2})=1$ for simplicity.

We will transmit the geometric information of $M$ to the sphere $S(p_0,1)$ such that the Laplacian $\Delta r$ of distance function $r(x)=d(p_0,x)$ along $S(p_0,1)$ almost equals that on the space form of constant curvature $\d$. Then by Lemma \ref{lem-quantitative-rigidity} below, (M2) holds.

\begin{lem}[Quantitative rigidity]\label{lem-quantitative-rigidity}
	Let $B(p_0,R)$ be a geodesic contractible ball in a complete Riemannnian manifold $N^{n+1}$ with $\mu\le K_N\le \d$. If the Laplacian of distance function $r(x)=d(p_0,x)$ satisfies
	\begin{equation}\label{ineq-quan-rig}
	\frac{1}{n}\Delta r\le \frac{c_\d}{s_\d}(R)+\epsilon
	\end{equation}
	along geodesic sphere $S(p_0,R)$, then $B(p_0,R)$ is $e^{C(n,\mu,\d,R) \sqrt{\epsilon}}$-almost isometric to a ball $B_\d(R)$ of constant curvature $\d$.
\end{lem}
\begin{proof}
	
	Let $J$ be a normal Jacobi field along a unit-speed radial geodesic $\gamma$ such that $J(0)=0$, $\gamma(0)=p_0$. Let $r$ be the distance function to $p$. Then $\dot{J}=\nabla_{\nabla r}J$ satisfies that
	$$\langle   \dot{J},X\rangle =\operatorname{Hess}r(J,X).$$
	Let $\rho(t)=\langle  \dot{J},J\rangle /|J|^2=\operatorname{Hess}r(J/|J|,J/|J|)$. Then by standard calculation for the Rauch comparison (see \cite[Theorem 6.4.3]{Petersen2016}),
	\begin{equation}\label{ineq-ode}
	-\d\le \dot{\rho}+\rho^2\le -\mu,
	\end{equation}

	By Taylor expansion $s_\mu(t)-s_\d(t)=\frac{\d-\mu}{6}t^3 +O(t^4)$,
	for $d_0=C(\mu,\d)\epsilon^{1/4}$, we have
	\begin{align*}
	\frac{|J(t)|}{s_\d(t)}\leq  \frac{s_\mu}{s_\d}(t) \leq e^{\sqrt{\epsilon}}, \quad 0<t\leq d_0.
	\end{align*}
	Thus $B(p_0,d_0)$ is $e^{\sqrt{\epsilon}}$-almost isometric to a $d_0$-ball of constant curvature $\d$.

	For $t\in [d_0,R]$, 
	let us consider  $\varphi(t):=\frac{c_\d}{s_\d}(t)$,  and the following auxiliary function
	\begin{equation}\label{func-quantitative-warped}
	F(t):=(\rho(t)-\varphi(t))e^{\int_{d_0}^t(\rho(s)+\varphi(s))ds}.
	\end{equation}
	By \eqref{ineq-ode},  $F(d_0)=\rho(d_0)-\varphi(d_0)\geq 0$.
	Since $\dot{\varphi}+\varphi^2=-\d$, by direct calculation, $\dot{F}(t)\ge 0$ for any $t\in (0,R]$.
	By (\ref{ineq-quan-rig}) and Jacobi comparison, 
	$$
	\begin{aligned}
	F(R) \le n \epsilon \cdot e^{\int_{d_0}^{R} \rho(s)+\varphi(s)ds}&=n\epsilon \frac{|J(R)|}{|J(d_0)|}\cdot \frac{s_\d(R)}{s_\d(d_0)}\\
	&\le n \epsilon \frac{s_\mu(R) s_\d(R)}{s_\mu(d_0)s_\d(d_0)} = C(n,\mu,\d,R)\sqrt{\epsilon}.
	\end{aligned}$$
	Then the monotonicity of $F$ implies $F(t)\le C(n,\mu,\d,R)\sqrt{\epsilon}$ for all $t\in [d_0,R]$. It follows that
	$$
	0\leq \rho(t)-\varphi(t) \leq  F(t)  \leq C(n,\mu,\d,R)\sqrt{\epsilon}, \quad \forall t\in [d_0,R].
	$$
	Finally, we have
	$$
	s_\d(t)\le |J(t)|\le e^{C(n,\mu,\d,R) \sqrt{\epsilon}} \cdot s_\d(t), \quad t\in [d_0,R].
	$$
\end{proof}

Let us continue the proof of (M2). Before transmitting geometric control of $M$ to $S(p_0,1)$, we make some preparation.

By the definition of $R_0=1$, we have $\|H\|_\infty=\frac{c_\d}{s_\d}(1)$. Together with \eqref{rescaling-inv-non-local-col}, we get
\begin{align}\label{uniform-upper-bound-area}
|M| \leq C(n,\d,A). 
\end{align} 

On $M$, by (\ref{main-theo-1.1}), we have
\begin{align}\label{r-Linfty-estimate}
\left|\frac{s_{\d}(r)}{s_\d(1)}-1 \right| \leq C(n,\d,R,A)\e^\frac{1}{2(2n+1)}, \quad \left|\frac{c_{\d}(r)}{c_\d(1)}-1 \right| \leq C(n,\d,R,A)\e^\frac{1}{2(2n+1)}.
\end{align}

Furthermore, based on an observation of Grosjean and Roth \cite[Proposition 2.1]{Grosjean-Roth2012} (see Lemma \ref{lem-l2-bound-perp}), we get
\begin{align}\label{nabla-r-L2-estimate}
\int_{M}|\nabla^{M}r|^2  \leq C(n,\d,A) \e.
\end{align}

In order to approximate $\Delta r$ on $M$, we refine inequality (\ref{3.1}) by Heintze into the following form.

\begin{lem}\label{lem-refine-heintze}
	Let the assumptions be as in Lemma \ref{lem-heintze}.
	\begin{equation}\label{eq-laplacian-mean-curv}
	\divv_MX^\top=s_\d \Delta r-nH\langle  X,\nu\rangle +\left[c_\d-s_\d \operatorname{Hess}r(e_n^*,e_n^*)\right]|\nabla^M r|^2,
	\end{equation}
	where $e_n^*$ is defined in proving (\ref{3.1}).
\end{lem}
\begin{proof}
	By direct calculation as in Lemma \ref{lem-heintze},
	$$\divv_MX=s_\d \Delta r+\left[c_\d-s_\d \operatorname{Hess}r(e_n^*,e_n^*)\right]|\nabla^M r|^2.$$
	Then (\ref{eq-laplacian-mean-curv}) follows by proceeding along the proof lines of (\ref{3.1}).
\end{proof}

Integrating (\ref{eq-laplacian-mean-curv}), by divergence theorem and \eqref{nabla-r-L2-estimate} we get
\begin{align}\label{refine-heintze-estimate}
\int_M (s_\d \D r-n H \langle X,\nu\rangle)=\int_M (s_\d \operatorname{Hess}r(e_n^\ast,e_n^\ast)-c_\d)|\nabla^M r|^2\leq C(n,\mu,\d,R,A) \e.
\end{align}
Thus, by \eqref{uniform-upper-bound-area}, \eqref{r-Linfty-estimate}, \eqref{refine-heintze-estimate} and $\D r \le n\frac{c_\mu}{s_\mu}(r)$, we have
\begin{align*}
     &\int_{M} \(\D r-n\frac{c_\d}{s_\d}(1)\) \\ 
\leq &\int_{M} \(\frac{s_\d(r)}{s_\d(1)}\D r-n\frac{s_\d(r)}{s_\d(1)}\|H\|_\infty\)+\int_{M} \(1-\frac{s_\d(r)}{s_\d(1)}\)\D r +n\|H\|_\infty \int_{M} \(\frac{s_\d(r)}{s_\d(1)}-1\)\\ 
\leq &\frac{1}{s_\d(1)}\int_{M} \(s_\d \D r-nH \langle X,\nu\rangle\)+\(\|\D r\|_\infty +n\|H\|_\infty\)|M| \left|1-\frac{s_\d(r)}{s_\d(1)}\right| \\
\leq &C(n,\mu,\d,R,A) \e^\frac{1}{2(2n+1)}.
\end{align*}
On the other hand, it follows from \eqref{r-Linfty-estimate} and $\D r \ge n\frac{c_\d}{s_\d}(r)$ that 
\begin{align*}
\int_{M} \(n\frac{c_\d}{s_\d}(1)-\D r\)  \leq \int_{M} n\(\frac{c_\d}{s_\d}(1)-\frac{c_\d}{s_\d}(r)\) \leq C(n,\mu,\d,R,A) \e^\frac{1}{2(2n+1)}.
\end{align*}
Finally, we conclude that
\begin{equation}\label{deviation-laplace}
\int_{M} \left|\frac{c_\d}{s_\d}(1)-\frac{1}{n}\Delta r \right| \leq C(n,\mu,\d,R,A)\e^\frac{1}{2(2n+1)}.
\end{equation}

Now we are ready to prove (M2).
\vspace{2mm}
\begin{proof}[Proof of (M2) in Main Theorem]
	~
	
	Let us prove that there is $\epsilon_1>0$ such that
	\begin{equation}\label{ineq-m2}
	\frac{1}{n}\Delta r\le \frac{c_\d}{s_\d}(1)+\epsilon_1
	\end{equation}
	holds along the geodesic sphere $S=S(p_0,1)$ and $\epsilon_1\to 0$ as $\epsilon\to 0$.
	
	Then by Lemma \ref{lem-quantitative-rigidity}, $B(p_0,1)$ is $e^{C(n,\mu,\d,R)\sqrt{\epsilon_1}}$-isometric to a ball of constant curvature $\d$. Furthermore, by Cheeger-Gromov's convergence (Theorem \ref{thm-ChGro-Convergence}) (M2) holds.
	
	We now prove (\ref{ineq-m2}) by dividing into 2 steps.
	
	Step 1. For any fixed $s>0$ and $x\in S$, there are points in $M$ lying $B(x,s)$ with $\frac{1}{n}\Delta r<\frac{c_\d}{s_\d}(1)+\e'$, where $\e'\to 0$ as $\e\ra 0$.
	
	Let us argue by contradiction. If there is $\e'>0$ such that all points in $M\cap B(x,s)$ satisfy $\frac{1}{n}\Delta r\ge \frac{c_\d}{s_\d}(1)+ \e'$. Let $y_0\in M$ be a nearest point to $x$, then
	$$\int_{M \cap B(x,s)}\left| \frac{c_\d}{s_\d}(1)-\frac{1}{n}\Delta r\right|\ge
	\int_{M \cap B(y_0,\frac{s}{2})}\left| \frac{c_\d}{s_\d}(1)-\frac{1}{n}\Delta r\right|
	\ge |M \cap B(y_0,\frac{s}{2})|\cdot \e'$$
	By Lemma \ref{lem-local-ball-volume-estimate} below, the volume $|M \cap B(y_0,\frac{s}{2})|$ admits a positive lower bound. This contradicts to (\ref{deviation-laplace}).
	
	Step 2. There is $\e''>0$ such that (\ref{ineq-m2}) holds for $\e''$ at any point in $S$, where $\e''\to 0$ as $\e\to 0$.
	
	Indeed, by Step 1 as $\e\to 0$ there exists $0<s(\e)\to 0$ such that for any point $x\in S$, we are able to take a point $y$ in $M$ such that $d(x,y)<s(\e)$ and (\ref{ineq-m2}) holds at $y$.
	If $x$ and $y$ lie in the same radial geodesic starting at $p_0$, then by $\mu \le K_N\le \d$ and standard Jacobi estimate, Step 2 holds for such $x$.
	
	In the following we assume that $x$ and $y$ lies in different radial directions at $p_0$.
	Let $\alpha(t)=\exp_{p_0}(v+tw)$, where $\exp_{p_0}(v)=x$ and $\exp_{p_0}(v+w)=y$. Let $\{e_i\}_{i=1}^{n+1}$ be an orthonormal frame along $\alpha$. Then we extend $e_i$ to be Jacobi fields along radial geodesics starting at $p_0$. By direct calculation,
	\begin{align*}
	\frac{d}{d t}\Delta r  = \frac{d }{d t}\sum_i\langle  \nabla^N_{\nabla^N r}e_i,e_i\rangle 
	= \sum_i\langle  \nabla^N_{\alpha'(t)}\nabla^N_{\nabla^N r}e_i,e_i\rangle 
	=\sum_i\langle   R^N(\alpha',\nabla^N r)e_i,e_i\rangle 
	\end{align*}
	Since the curvature operator $R^N$ is bounded by $\frac{2}{3}(\d-\mu)$ and $|\alpha'|\le s_\d(1)\cdot d(x,y)$, we get
	$$
	\left|\frac{d}{dt}\Delta r\right|\le \frac{2}{3}(\d-\mu)s_\d(1)\cdot s.
	$$
	It follows that
	$$\Delta r(y)\le \Delta r(x) +\frac{2}{3}(\d-\mu)s_\d(1)\cdot s.$$
	Since by Step 1, $s(\e)\to 0$ as $\e\to 0$, the proof of Step 2 is completed.
\end{proof}

The following extrinsic non-collapsing property of $M$, which can be traced back to \cite[Proposition 1.12]{Colding-Minicozzi-book}, is important to Step 1 above in (M2)'s proof.

\begin{lem}\label{lem-local-ball-volume-estimate}
Let $M$ be a properly immersed hypersurface in a complete Riemannian manifold $N$ with $K_N\leq \d$. Let $x_0\in M$, and $0< r_0\le \min\left\{\inj_N (x_0), \|r\|_\infty, \frac{\pi}{2\sqrt{\d}}\right\}$, where $r:M\to \mathbb R$ is the distance to $x_0$ in $N$. If $\|H\|_\infty\leq \L$, for any $s\in (0,r_0)$ we have
\begin{align*}
\frac{|B(x_0,s) \cap M|}{\omega_n s_\d(s)^n} \leq e^{(n\L +\sqrt{\max\{-\d,0\}})r_0} \frac{|B(x_0,r_0) \cap M|}{\omega_n s_\d(r_0)^n}.
\end{align*}
In particular, letting $s\ra 0$ we have
\begin{align*}
|B(x_0,r) \cap M| \ge e^{-(n\L +\sqrt{\max\{-\d,0\}})r}\omega_n s_\d^n(r), \quad \forall r\in(0,r_0].
\end{align*}
\end{lem}
\begin{proof}
By \eqref{3.1}, we have
\begin{align*}
\operatorname{div}_M(X^\top) \geq n c_\d -n H\langle X,\nu \rangle \geq n c_\d - n|H|s_\d,
\end{align*}
which gives
\begin{align}\label{monotonicity-formula-1}
n \leq \frac{\divv_M(X^\top)}{c_\d}+\frac{n|H|s_\d}{c_\d}
= \divv_M \(\frac{X^\top}{c_\d}\)-\frac{\d s_\d^2 |\nabla^M r|^2}{c_\d^2}+\frac{n|H|s_\d}{c_\d}.
\end{align}
By the Stokes' theorem, we have
\begin{align}\label{monotonicity-formula-2}
n \Vol(\{r\le s\}) \le &\int_{\{r\le s\}} \divv_M \(\frac{X^\top}{c_\d}\)+\int_{\{r\le s\}} \(-\frac{\d s_\d^2 |\nabla^M r|^2}{c_\d^2}+\frac{n|H|s_\d}{c_\d} \) \nonumber \\
=&\int_{\{r=s\}} \left\langle \frac{X^\top}{c_\d}, \frac{\nabla^M r}{|\nabla^M r|} \right\rangle +\int_{\{r\le s\}} \(-\frac{\d s_\d^2 |\nabla^M r|^2}{c_\d^2}+\frac{n|H|s_\d}{c_\d} \) \nonumber \\
=&\frac{s_\d(s)}{c_\d(s)} \int_{\{r=s\}}|\nabla^M r| +\int_{\{r\le s\}} \( -\frac{\d s_\d^2 |\nabla^M r|^2}{c_\d^2}+\frac{n|H|s_\d}{c_\d} \).
\end{align}
The coarea formula implies that
\begin{align}\label{coarea-formula}
\Vol(\{r\leq s\})=\int_0^s \int_{\{r=s\}}|\nabla^M r|^{-1}.
\end{align}
Combining \eqref{monotonicity-formula-2} and \eqref{coarea-formula}, we have
\begin{align*}
    \frac{d}{ds}\(s_\d^{-n}(s)\Vol(\{r\le s\})\)
\ge & -s_\d^{-n}(s)\int_{\{r=s\}}\(|\nabla^M r|-|\nabla^M r|^{-1}\) \\
    &-s_\d^{-n-1}(s)c_\d(s)\int_{\{r\le s\}} \( -\frac{\d s_\d^2 |\nabla^M r|^2}{c_\d^2}+\frac{n|H|s_\d}{c_\d} \)\\
\ge & \d s_\d^{n-1}(s)c_\d(s)\int_{\{r\le s\}} \frac{s_\d^2 |\nabla^M r|^2}{c_\d^2} -n\L s_\d^{-n-1}(s)c_\d(s) \int_{\{r\le s\}} \frac{s_\d}{c_\d}.
\end{align*}
where we used the fact $|\nabla^M r|\le 1$. If $\d \ge 0$, then $\frac{s_\d}{c_\d}(r)\le \frac{s_\d}{c_\d}(s)$. This gives
\begin{align*}
\frac{d}{ds}\(s_\d^{-n}(s)\Vol(\{r\le s\})\)
\ge & -n\L s_\d^{-n-1}(s)c_\d(s) \int_{\{r\le s\}} \frac{s_\d}{c_\d} \\
\ge & -n\L s_\d^{-n}(s)\Vol(\{r\le s\}).
\end{align*}
If $\d<0$, then $\frac{s_\d}{c_\d}(r)\le \frac{s_\d}{c_\d}(s) \le \frac{1}{\sqrt{|\d|}}$. It follows that
\begin{align*}
\frac{d}{ds}\(s_\d^{-n}(s)\Vol(\{r\le s\})\)
\ge & -|\d|s_\d^{n-1}(s)c_\d(s)\int_{\{r\le s\}} \frac{s_\d^2 |\nabla^M r|^2}{c_\d^2} -n\L s_\d^{-n-1}(s)c_\d(s) \int_{\{r\le s\}} \frac{s_\d}{c_\d}\\
\ge & (-\sqrt{|\d|}-n\L)\(s_\d^{-n}(s)\Vol(\{r\le s\})\)
\end{align*}
Combining these estimates together, we obtain
\begin{align*}
\frac{d}{ds}\(s_\d^{-n}(s)\Vol(\{r\le s\})\) \ge (-\sqrt{\max\{0,-\d\}}-n\L)\(s_\d^{-n}(s)\Vol(\{r\le s\})\).
\end{align*}
Let $F(s)=s_\d^{-n}(s)\Vol(\{r\le s\}$, then for any $s\in(0,r_0)$ we have
\begin{align*}
F(s) \leq F(r_0) e^{(n\L+\sqrt{\max\{0,-\d\}})(r_0-s)}\le F(r_0)e^{(n\L+\sqrt{\max\{0,-\d\}})(r_0)},
\end{align*}
which completes the proof.
\end{proof}

\section{1st eigenvalue pinching implies fixed position}
\label{sec:4}

In this section we prove Theorem \ref{main-theo-1}.

Let $p_0$ be the center of mass of $M$ and $r(x)=d(x,p_0)$, and let $X=s_\d(r)\nabla^N r$ be the position vector with respect to $p_0$ and $X^\top$ be its tangential projection over $M$.

\subsection{$L^2$-pinching of position vector}

Let us first transform  (\ref{pinching-condition-1}) to a pinching condition more related to the position vector $X$, which is due to Grosjean and Roth \cite[Proposition 2.1]{Grosjean-Roth2012}.

\begin{lem}[\cite{Grosjean-Roth2012}]\label{lem-tranform-pinching}
	Assume that $M$ is a compact immersed hypersurface in a convex ball $B(p,R)$ of $N$, where $K_N\leq \d$. (For $\d>0$, we assume in addition $R\le \frac{\pi}{8\sqrt{\d}}$.)
	Then
	\begin{align}\label{4.5}
	1\leq (\d+\|H\|_\infty^2)\|X\|_2^2.
	\end{align}
	Moreover, for any $0\le \e\le \frac{1}{2}$, (\ref{pinching-condition-1}) implies
	\begin{equation}\label{ineq-Ie}
	(\d+\|H\|_\infty^2)\|X\|_2^2 \leq 1+4\e.
	\end{equation}
	
\end{lem}

By Lemma \ref{lem-tranform-pinching}, we will use the following $L^2$-pinching condition in remaining subsections instead of (\ref{pinching-condition-1}).
\begin{equation}\label{l2-pinching}
1\le (\d+\|H\|_\infty^2)\|X\|_2^2 \leq 1+\e
\end{equation}

Since some estimates in the proof of Lemma \ref{lem-tranform-pinching} will be used later, for completeness we give its proof.
\begin{proof}[Proof of Lemma \ref{lem-tranform-pinching}]
	~
	
	We first prove the first inequality (\ref{4.5}). By divergence theorem and (\ref{3.1}),
	\begin{align}
	\int_M \(n-\d|X^\top|^2\)& \leq \int_M \(n-\divv(X^\top) c_\d\) \notag\\
	& \leq \int_M \(n-n c^2_\d+nH\langle X,\nu \rangle  c_\d\) \quad (\text{by putting (\ref{3.1}) into above})\notag\\
	& \leq  \int_M n\d  s_\d^2+\|H\|_\infty\int_M n  s_\d  c_\d \quad (\text{by $c^2_\d+\d s_\d^2=1$}) \notag\\
	&\leq  n\d\int_M  s_\d^2+\|H\|_\infty\int_M \(nH\langle X,\nu\rangle s_\d+\divv(X^\top) s_\d \) \quad (\text{by (\ref{4.5}) again}) \label{calculate-1}
	\end{align}
	By divergence theorem, $\int_M \divv(X^\top)s_\d=-\int_M c_\d \langle  \nabla^N r, X^\top\rangle $. Then by using $|X|=s_\d$, we derive
	\begin{align}
	(\ref{calculate-1})
	& \leq  n\d\int_M |X|^2+n\|H\|_\infty^2\int_M |\langle X,\nu\rangle||X|-\int_M \|H\|_\infty\frac{ c_\d}{ s_\d}|X^\top|^2 \notag\\
	& \leq n(\d+\|H\|_\infty^2)\int_M |X|^2-\int_M \frac{ c_\d}{ s_\d}\|H\|_\infty|X^\top|^2, \label{4.2}
	\end{align}
	which can be transformed to
	\begin{align}\label{4.3}
	1\leq (\d+\|H\|_\infty^2)\|X\|_2^2+\frac{1}{n|M|}\int_M \(\d-\frac{ c_\d}{ s_\d} \|H\|_\infty\) |X^\top|^2.
	\end{align}
	
	Note that $M\subset B(p,R)\subset B(p_0,2R)$, where $K_N\le \delta$.
	
	If $\d \leq 0$, then (\ref{4.5}) follows from \eqref{4.3}, since the latter term in RHS of (\ref{4.3}) is non-positive.
	
	If $\d>0$, in order to estimate the latter term in RHS of \eqref{4.3}, we further require $R\le \frac{\pi}{8\sqrt{\d}}$. Then by \eqref{lower-bound-mean-curv} and Rauch comparison, we have
	\begin{align*}
	\min\left\{\|H\|_\infty,\frac{c_\d}{s_\d}(r(x))\right\}\geq \frac{ c_\d}{ s_\d}(2R)\geq \frac{c_\d}{s_\d}\(\frac{\pi}{4\sqrt{\d}}\)=\sqrt{\d}, \quad \text{for all $x\in M$}.
	\end{align*}
	Thus we derive
	\begin{align}\label{4.6}
	\d-\frac{c_\d}{ s_\d}(r(x)) \|H\|_\infty\leq \d-\sqrt{\d}\|H\|_\infty=\sqrt{\d}(\sqrt{\d}-\|H\|_\infty)\leq 0, \quad \text{for all $x\in M$},
	\end{align}
	And (\ref{4.5}) follows from (\ref{4.6}) and (\ref{4.3}).
	
	Next, let us prove (\ref{pinching-condition-1}) implies the 2nd inequality (\ref{ineq-Ie}).
	
	By multiplying $\|X\|_2^2$ to both side of $(\ref{pinching-condition-1})$, we derive
	\begin{align}
	n(\d+\|H\|_\infty^2)\|X\|_2^2 &\leq (1+\e)\l_1(M)\|X\|_2^2 \notag\\
	& \leq (1+\e)\|\nabla^M X\|_2^2 \quad (\text{by Rayleigh quotient}).\label{4.7}
	\end{align}
	Since in normal coordinates, $X=(\frac{s_\d(r)}{r}x_i)_{i=1}^{n+1}$, combining
	(\ref{3.2}) and (\ref{4.7}) we have
	\begin{equation}\label{est-l2-pinching-1}
	n(\d+\|H\|_\infty^2)\|X\|_2^2 \leq  (1+\e)(n-\d\|X^\top\|_2^2)
	\end{equation}
	
	For $\d \geq 0$, (\ref{est-l2-pinching-1}) immediately implies (\ref{ineq-Ie}).

	If $\d<0$, it is necessary to estimate $-\d \|X^\top \|_2^2$. By (\ref{4.2}) we have
	\begin{align}\label{4.8}
	n-\d\|X^\top\|_2^2 \leq n(\d+\|H\|_\infty^2)\|X\|_2^2-\frac{\|H\|_\infty}{|M|}\int_M \frac{ c_\d}{ s_\d}|X^\top|^2.
	\end{align}
	by reformulating (\ref{4.8}), we have
	\begin{align}
	\frac{\|H\|_\infty}{|M|}\int_M \frac{ c_\d}{ s_\d}|X^\top|^2 &\leq n(\d+\|H\|_\infty^2)\|X\|_2^2-(n-\d\|X^\top\|^2)\nonumber\\
	&\leq \frac{n\e}{1+\e}(\d+\|H\|_\infty^2)\|X\|_2^2. \quad \text{(by puting (\ref{est-l2-pinching-1}) into above)} \label{ineq-l2-pinching-a}
	\end{align}
		
	On the other hand, by the fact
	$$\min\left\{\|H\|_\infty, \frac{c_\d}{s_\d}(r(x))\right\} \geq \sqrt{-\d}  \quad\text{for $x\in M$},$$ it is clear
	\begin{equation}\label{int-lower-bound-1}
	\frac{\|H\|_\infty}{|M|}\int_M \frac{ c_\d}{ s_\d}|X^\top|^2\ge -\d\|X^\top\|_2^2.
	\end{equation}
	
	By combining (\ref{ineq-l2-pinching-a}) and (\ref{int-lower-bound-1}),
	\begin{align}\label{ineq-l2-pinching-b}
	-\d\|X^\top\|_2^2\leq\frac{n\e}{1+\e}(\d+\|H\|_\infty^2)\|X\|_2^2.
	\end{align}
	Put (\ref{ineq-l2-pinching-b}) into (\ref{est-l2-pinching-1}) (or the middle of (\ref{ineq-l2-pinching-a})), we derive
	\begin{align*}
	(\d+\|H\|_\infty^2)\|X\|_2^2 \leq \frac{1+\e}{1-\e} \leq 1+4\e,
	\end{align*}
	where the last inequality holds for $0\le \e\le \frac{1}{2}$,
\end{proof}

An easy corollary of Lemma \ref{lem-tranform-pinching} is an $L^2$-bound of the tangential component $X^\top$ of position vector $X$, which has been used in the proof of (M2).
\begin{lem}\label{lem-l2-bound-perp}
	Let $M$ be as in Main Theorem. If (\ref{l2-pinching}) holds, then
	\begin{equation}
	\|X^\top\|_2^2 \leq \frac{2\e}{\|H\|_\infty^2},\label{l2-bound-perp-position}
	\end{equation}
\end{lem}
\begin{proof}
	By (\ref{4.2}), (\ref{4.6}) and (\ref{l2-pinching}), we have
	\begin{align*}
	\frac{\|H\|_\infty^2}{|M|} \int_M \(|X|^2-|\langle X,\nu\rangle| |X|\)
	\leq &(\d+\|H\|_\infty^2)\|X\|_2^2-1+\frac{1}{n|M|}\int_M \(\d-\frac{c_\d}{ s_\d}\|H\|_\infty\)|X^\top|^2 \\
	\leq &(\d+\|H\|_\infty^2)\|X\|_2^2-1\\
	\leq &\e.
	\end{align*}
	Thus, we obtain
	\begin{align*}
	\|X^\top\|_2^2=&\frac{1}{|M|}\int_M \(|X|^2-|\langle X,\nu\rangle|^2\)
	\leq \frac{2}{|M|}\int_M \(|X|^2-|X||\langle X,\nu\rangle|\)
	\leq \frac{2\e}{\|H\|_\infty^2}.
	\end{align*}
\end{proof}

\subsection{Estimate on the position vector}
Under the assumption (\ref{rescaling-inv-non-local-col}) we prove (\ref{main-theo-1.1}) in this subsection by following the main ideas of \cite{Colbois-Grosjean2007}.

Note that, via introducing the radius $R$ of extrinsic ball $B(p,R)$ in our estimates, we improve those in \cite{Grosjean-Roth2012} (cf. Lemmas 3.3, 3.5, etc. in \cite{Grosjean-Roth2012}) with simpler proofs, such that their technical assumption $\frac{\|H\|_\infty}{\sqrt{\d+\|H\|_\infty^2}}\le A_1$  in \cite{Grosjean-Roth2012} is dropped in our proof.

For simplicity, let $h:=\sqrt{\d+\|H\|_\infty^2}$. Let us consider an important auxiliary function
\begin{equation}\label{aux-func}
\psi=|X|^\frac{1}{2}\left| |X|-\frac{1}{h} \right|,
\end{equation}
We will estimate its $L^\infty$-norm.

Let us rewrite
\begin{equation}\label{4.10}
\begin{split}
\psi=&\frac{|X|^\frac{1}{2}}{h^2}\left|(h^2 X-\d X-H c_\d \nu)+\(\d X+H c_\d \nu-h\frac{X}{|X|}\)\right|\leq \frac{|X|^\frac{1}{2}}{h^2}|Y|+\frac{|W|}{h^2},
\end{split}
\end{equation}
where $Y=h^2 X-\d X-H c_\d \nu=\|H\|_\infty^2 X-H c_\d \nu$ and $W=|X|^\frac{1}{2}\(\d X+H c_\d \nu-h\frac{X}{|X|}\)$.

In order to estimate $\|\psi\|_\infty$, the following $L^2$-bounds are necessary.

\begin{lem}\label{lem-estimate-psi-1}
	The $L^2$-pinching condition (\ref{l2-pinching})
	with $\e<1$ implies
	\begin{align}
	\|Y\|_2^2 & \leq C(n,\d,R)h^2\e \label{ineq-upper-bound-Y}\\
	\|W\|_2^2 & \leq C(n,\d,R,A_1)h \e \label{ineq-upper-bound-W-final}
\end{align}
\end{lem}
An upper bound of $\|X\|_\infty$, which is by a Nirenberg-Moser type iteration, is also used in estimating $\|\psi\|_\infty$.

\begin{lem}\label{lem-infity-bound-position}
	If $|M|^\frac{1}{n}\|H\|_\infty\leq A_1$, then
	\begin{align}\label{ineq-upper-bound-infinity-position}
	\|X\|_\infty\leq C(n)A_1^{\frac{n}{2}} \|X\|_2.
	\end{align}
\end{lem}

The proofs of Lemmas \ref{lem-estimate-psi-1} and \ref{lem-infity-bound-position} will be given later in the appendix.

Let us continue the proof of Theorem \ref{main-theo-1}.

\begin{lem}\label{lem-infinity-bound-psi}
	The $L^2$-pinching (\ref{l2-pinching}) with $\e<1$ and $|M|^\frac{1}{n}\|H\|_\infty\leq A_1$ implies that
	\begin{equation}\label{ineq-upper-bound-infinity-psi}
	\|\psi\|_\infty \leq \frac{C(n,\d,R,A_1)}{h^\frac{3}{2}}\e^\frac{1}{2(2n+1)}.
	\end{equation}
\end{lem}
\begin{proof}
	Firstly, by (\ref{4.10}) and Lemma \ref{lem-estimate-psi-1}, we have the following $L^1$-bound of $\psi$,
	\begin{align*}
	\|\psi\|_1 \leq \frac{1}{h^2} \(\|X\|_2^\frac{1}{2}\|Y\|_2+\|W\|_2\)
	\leq \frac{C(n,\d,R,A_1)}{h^\frac{3}{2}}\e^\frac{1}{2}.
	\end{align*}
	
	Secondly, let us apply Nirenberg-Moser type iteration to estimate $\|\psi\|_\infty$ as in below.
	
	By the upper bound (\ref{ineq-upper-bound-infinity-position}) of $\|X\|_\infty$, for any $\a\geq 1$, we have
	\begin{align*}
	|\mrm{d}\psi^{2\a}|=\a \psi^{2\a-2} |\mrm{d}\psi^2|
	&\leq  \a \psi^{2\a-2} \left||X|-\frac{1}{h} \right| \left| 3|X|-\frac{1}{h} \right| \left|\,\mrm{d} |X|\,\right| \\
	&\leq  3\a \psi^{2\a-2} \(\|X\|_\infty+\frac{1}{h}\)^2  c_\d \\
	&\leq  3\a \psi^{2\a-2} \(\|X\|_\infty+\frac{1}{h}\)^2 (1+\|H\|_\infty \|X\|_\infty)\\
	&\leq  C(n,A_1)\a \psi^{2\a-2}\|H\|_\infty\frac{1}{h^3},
	\end{align*}
	and
	\begin{align*}
	\|\psi \|_\infty^2 \leq \|X\|_\infty\(\|X\|_\infty+\frac{1}{h}\)^2\leq C(n,A_1)\frac{1}{h^3}.
	\end{align*}
	Applying Sobolev inequality (\ref{Hoffman-Spruck}) to $f=\psi^{2\a}$ with $\a\geq 1$, we have
	\begin{align*}
	\(\frac{1}{|M|}\int_M \psi^{\frac{2 \a n}{n-1}} \) & \leq C(n)|M|^\frac{1}{n} \left[ \frac{1}{|M|}\int_M \( |\mrm{d}\psi^{2\a}|+|H|\psi^{2\a} \) \right]\\
	& \leq C(n)|M|^\frac{1}{n} \left[ \frac{1}{|M|}\int_M \( C(n,A_1)\a \psi^{2\a-2}\|H\|_\infty\frac{1}{h^3}+\|H\|_\infty \|\psi\|_\infty^2 \psi^{2\a-2} \) \right]\\
	& \leq C(n,A_1) \a \frac{1}{h^3}\( \frac{1}{|M|}\int_M \psi^{2\a-2} \),
	\end{align*}
	which gives
	\begin{align*}
	\|\psi\|_{\frac{2\a n}{n-1}}^{2\a} \leq & C(n,A_1) \a \frac{1}{h^3} \|\psi\|_{2\a-2}^{2\a-2}.
	\end{align*}
	We take $\tau=\frac{n}{n-1}$ and $\a=\frac{b_i}{2}+1$, where $b_{i+1}=(b_i+2)\tau$ and $b_0=1$, then
	$$
	b_i=(1+2n)\tau^{i}-2n\geq 1, \quad \text{for $i=0,1,\cdots$,}
	$$
	and $\frac{2\a n}{n-1}=2\a \tau= (b_i+2)\tau=b_{i+1}$. Then we obtain
	\begin{align*}
	\|\psi\|_{b_{i+1}}^{\frac{b_{i+1}}{\tau}} \leq & C(n,A_1) \(\frac{b_i}{2}+1\) \frac{1}{h^3} \|\psi\|_{b_i}^{b_i} \leq C(n,A_1)b_i\frac{1}{h^3} \|\psi\|_{b_i}^{b_i}.
	\end{align*}
	By iteration, we get
	\begin{align*}
	\|\psi\|_{b_{i+1}}^{\frac{b_{i+1}}{\tau^{i+1}}} \leq & \left[ C(n,A_1) b_i \frac{1}{h^3}\right]^{\frac{1}{\tau^i}}\|\psi\|^{\frac{b_i}{\tau^i}}_{b_i}
	\leq \(\prod_{k=0}^{i} b_k^{\frac{1}{\tau^k}}\)\left[ C(n,A_1)\frac{1}{h^3} \right]^{n\(1-\frac{1}{\tau^{i+1}}\)} \|\psi\|_{b_0}^{b_0}.
	\end{align*}
	Note that $\prod_{k=0}^{\infty} b_k^{\frac{1}{\tau^k}}$ converges to a constant $C(n)$ by the ratio test. Since $\frac{b_i}{\tau^i}$ converges to $1+2n$, we have
	\begin{align*}
	\|\psi\|_\infty^{1+2n} \leq \left[ C(n,A_1)\frac{1}{h^3} \right]^{n} \|\psi\|_1 \leq \frac{C(n,\d,R,A_1)}{h^{3n+\frac{3}{2}}}\e^\frac{1}{2},
	\end{align*}
	which implies (\ref{ineq-upper-bound-infinity-psi}).
\end{proof}

Now we are ready to prove Theorem \ref{main-theo-1}.
\vspace{2mm}
\begin{proof}[Proof of Theorem \ref{main-theo-1}]
	~
	
	Let $f(t):=t(t-\frac{1}{h})^2$. By Lemma \ref{lem-infinity-bound-psi} above we have
	\begin{align*}
	\|f(|X|)\|_\infty=\|\psi\|_\infty^2 \leq \frac{C(n,\d,R,A_1)}{h^3}\e^\frac{1}{2n+1}.
	\end{align*}
	If we pick $\e_0$ sufficiently small such that $C(n,\d,R,A_1)\e_0^\frac{1}{2n+1}<\frac{1}{27}$, then $\|f(|X|)\|_\infty\leq \frac{1}{27h^3}<f(\frac{1}{3h})$. Since $\|X\|_2^2 \geq \frac{1}{h^2}$, there exists a point $x_0\in M$ such that $|X_{x_0}|\geq \frac{1}{h}>\frac{1}{3h}$. By the connectedness of $M$, it follows that $|X|>\frac{1}{3h}$ on $M$. Thus, if (\ref{l2-pinching}) holds with $\e<\e_0$, then
	\begin{align*}
	\frac{1}{3h} \left| |X_x|-\frac{1}{h} \right|^2 \leq \|f(|X|)\|_\infty\leq \frac{C(n,\d,R,A_1)}{h^3}\e^\frac{1}{2n+1},
	\end{align*}
	which gives (\ref{main-theo-1.1}).
\end{proof}

\begin{rem}\label{rem-fixed-position}
	By the proofs in this section, it is clear that the validity of (\ref{main-theo-1.1}) does not depend on the lower bound of the sectional curvature of the ambient space.
	
	Moreover, the proofs of (M1) and (M2) may go through once $N$ satisfies $C^{1,\alpha}$-convergence regularity. Thus it is likely that the condition $K_N\ge \mu$ can be weakened.
\end{rem}

\section{Proof of Theorem \ref{main-theo-B}}
\label{sec:5}

In this section we prove Theorem \ref{main-theo-B}.

We will prove that, under the assumptions of Theorem \ref{main-theo-B}, there exist positive functions $$\e_2=\e_2(A_1,A_2,q,R,\d,\mu,n)>0$$ such that, pinching condition (\ref{pinching-condition-1}) with $\e<\e_2$ implies that the radial projection $F$ from $M$ to sphere $S(p_0,R_0)$ is a diffeomorphism, and satisfies
\begin{equation}\label{ineq-almost-isom}
\left| |dF_x(u)|^2-1 \right|\leq C_2\e^{\min\{ \frac{1}{2(2n+1)},\frac{q-n}{2(q-n+qn)}\}}, \quad \text{for any $x\in M$ and any unit vector $u\in T_x M$}
\end{equation}
where $C_2=C_2(n,q,\mu,\d,R,A_1,A_2)$ is a universal constant.

Let $M$ and $N$ be as in Theorem \ref{main-theo-B}.
Let $S(p_0,R_0)$ be the geodesic sphere centered at $p_0$ in $N$ with $R_0=s_\d^{-1}(\frac{1}{h})$. The radial projection is defined by
$$
F: M \to S\(p_0,R_0\)\subset N, x\mapsto \exp_{p_0}\(R_0\frac{Y}{|Y|}\),
$$
where $Y=\exp_{p_0}^{-1}(x)$. Let $\varrho:=R_0\frac{Y}{|Y|}$ for simplicity. By direct observation, the differential of $F$ satisfies the following.
\begin{lem}[{\cite[Lemma 4.1]{Grosjean-Roth2012}}]
	Let $u\in U_x M$ and $v=u-\langle u,\nabla^N r\rangle \nabla^N r$. Then
	\begin{align*}
	|dF_x(u)|^2=\frac{R_0^2}{r^2}\left| d\exp_{p_0}|_\varrho(d\exp_{p_0}^{-1}|_x(v)) \right|^2.
	\end{align*}
\end{lem}

The main technical lemma in proving (\ref{ineq-almost-isom}) is below.

\begin{lem}\label{lem-almost-isometry}
	Let the assumptions be as in Main Theorem.
	\begin{enumerate}
		\item\label{est-df} There exists a positive constant $C(A_1,R,\d,\mu,n)$ such that for any $u\in T_x M$ with $|u|=1$,
		\begin{align*}
		\left| |dF_x(u)|^2-1 \right| \leq C(n,\mu,\d,R,A_1)\e^\frac{1}{2(2n+1)}+\|\nabla^M r\|_\infty^2.
		\end{align*}
		\item\label{est-grad-r} If  $|M|^\frac{1}{n}\|B\|_q\leq A_2$ with $q>n$, then there exists a positive constant $\e_0(A_1,R,\d,n)$ such that $L^2$-pinching (\ref{l2-pinching}) with $0\le \e<\e_0$ implies that
		\begin{align*}
		\|\nabla^M r \|_\infty \leq C\e^\frac{q-n}{2(q-n+qn)},
		\end{align*}
		where $C=C(A_1,A_2,q,R,\d,\mu,n)$ is a universal constant.
	\end{enumerate}
\end{lem}

\vspace{2mm}
\begin{proof}[Proof of Theorem \ref{main-theo-B}]
	~
	
	By Lemma \ref{lem-tranform-pinching}, we may assume (\ref{l2-pinching}).
	
	Combining (\ref{est-df}) and (\ref{est-grad-r}), we have
	\begin{align*}
	\left| |dF_x(u)|^2-1 \right| \leq C(n,\mu,\d,R,A_1)\e^\frac{1}{2(2n+1)}+C(n,q,\mu,\d,R,A_1,A_2)\e^\frac{q-n}{2(q-n+qn)}.
	\end{align*}
	By picking $\e_0=\e_0(n,q,\mu,\d,R,A_1,A_2)$ sufficiently small,  (\ref{ineq-almost-isom}) follows. Thus the map $F$ is a local diffeomorphism. Since $S(p_0,R_0)$ is simply connected for $n\geq 2$ and $M$ is closed, $F$ is a diffeomorphism.
	
	By Gauss equation, $M^n$ admits $L^{q/2}$-integral Ricci curvature bound with $q>n$. Since $dF$ has norm almost $1$, $M$ also is of $e^{\varkappa(\e)}$-non-collapsing at some scale  $r_0$ proportional to $R_0$. Therefore, by the $C^{\alpha}$-convergence Theorem \ref{thm-Calpha-convergence}, $M$ is $C^{\alpha}$-close to a round sphere of constant curvature.
	
\end{proof}

What remains in this section is the proof of Lemma \ref{lem-almost-isometry}.

Compared to the corresponding situation in \cite{Grosjean-Roth2012}, we already obtain
a pointwise control on the norm of position vector (i.e., (\ref{main-theo-1.1})) in a large ball $B(p,R)$ via improved estimates, such that $M$ lies in the $\e$-tubular neighborhood of $S(p_0,R_0)$. Based on this fact, we apply Jacobi comparison to further improve \cite[Lemma 4.2]{Grosjean-Roth2012} into the form in Lemma \ref{lem-almost-isometry}.

\vspace{2mm}
\begin{proof}[Proof of (\ref{est-df}) in Lemma \ref{lem-almost-isometry}]
	~
	
	Let $Y=\exp_{p_0}^{-1}(x)$. Let $v=u-\langle u,\nabla^N r\rangle \nabla^N r$ and $V=d\exp_{p_0}^{-1}|_x(v)$, then $V\bot Y$. Let $c$ be the geodesic ray from $p$ through $x$. Note that $v\in T_x N$ such that $\langle v,\nabla^N r\rangle_x=0$. To estimate $|dF_x(u)|^2$, we define the Jacobi field as follows:
	$$\left\{
	\begin{aligned}
	&J''(t)+R(J(t),\dot{\g}(t))\dot{\g}(t)=0, \\
	&J(0)=0, \quad J'(0)=\frac{V}{|Y|}.
	\end{aligned}
	\right.
	$$
	Then the Jacobi field can be explicitly expressed by
	\begin{align*}
	J(t)=d\exp_{p_0}|_{t\frac{Y}{|Y|}}\(t\frac{V}{|Y|}\).
	\end{align*}
	Observe that
	\begin{align*}
	J(r(x))= &J(|Y|)=d\exp_{p_0}|_{Y}\(V\)=v,\\
	J(R_0)= &d\exp_{p_0}|_{R_0\frac{Y}{|Y|}}\(R_0\frac{V}{|Y|}\)=dF_x(u).
	\end{align*}
	Take $s_1=\min\{r(x),R_0\}$ and $s_2=\max\{r(x),R_0\}$. If $\d>0$, $M \subset B(p,R)\subset B(p_0,2R)$ implies that $r(x)\leq 2R\leq \frac{\pi}{4\sqrt{\d}}$. On the other hand, $h \geq \frac{1}{ s_\d(2R)}$ implies that $R_0= s_\d^{-1}(\frac{1}{h}) \leq 2R \leq \frac{\pi}{4\sqrt{\d}}$. Hence $0\leq s_1\leq s_2 \leq \frac{\pi}{4\sqrt{\d}}$ if $\d>0$. By the standard Jacobi field estimates on $N$ with $\mu \leq K_N \leq \d$, we have
	\begin{align*}
	\frac{ s_\mu(s_1)}{ s_\mu(s_2)}\leq \left|\frac{J(s_1)}{J(s_2)}\right| \leq \frac{s_\d(s_1)}{s_\d(s_2)},\quad \text{for} \quad 0< s_1 \leq s_2<\frac{\pi}{\sqrt{\d}}.
	\end{align*}
	\begin{enumerate}[(1)]
		\item If $r(x)\leq R_0$, then
		\begin{align*}
		\(\frac{ s_\mu(r(x))}{ s_\mu(R_0)} \)^2\leq \frac{|v|^2}{|dF_x(u)|^2} \leq \(\frac{s_\d(r(x))}{s_\d(R_0)}\)^2.
		\end{align*}
		\item If $r(x)> R_0$, then
		\begin{align*}
		\(\frac{ s_\mu(R_0)}{ s_\mu(r(x))} \)^2\leq \frac{|dF_x(u)|^2}{|v|^2} \leq \(\frac{s_\d(R_0)}{s_\d(r(x))}\)^2.
		\end{align*}
	\end{enumerate}
	Combining these two estimates, we obtain
	\begin{align*}
	\min \left\{\frac{1}{h s_\d(r(x))},\frac{ s_\mu(R_0)}{ s_\mu(r(x))} \right\}\leq \frac{|dF_x(u)|^2}{|v|^2} \leq \max \left\{\frac{1}{h s_\d(r(x))},\frac{ s_\mu(R_0)}{ s_\mu(r(x))} \right\}.
	\end{align*}
	By (\ref{main-theo-1.1}), we have
	$$
	1-C(n,\d,R,A_1)\e^{\frac{1}{2(2n+1)}} \leq h s_\d (r(x)) \leq 1+C(n,\d,R,A_1)\e^{\frac{1}{2(2n+1)}}, \quad \text{for all $x\in M$}.
	$$
	Note that the Lipschitz constant of $s_\mu\circ s_\d^{-1}$ depends only on $\mu$ and $\d$, we have
	\begin{align*}
	&\left| |dF_x(u)|^2-1 \right| \\
	\leq &\max\left\{ \left|\(1+C(n,\mu,\d,A_1,R)\e^\frac{1}{2(2n+1)}\)|v|^2-1\right|,\left|\(1-C(n,\mu,\d,A_1,R)\e^\frac{1}{2(2n+1)}\)|v|^2-1\right| \right\}\\
	\leq &C(n,\mu,\d,R,A_1)\e^\frac{1}{2(2n+1)}+\|\nabla^M r\|_\infty^2,
	\end{align*}
	where we have used the fact $1-|\nabla^M r|^2 \leq |u|^2-\langle u,\nabla^M r \rangle^2=|v|^2 \leq 1$.
\end{proof}

\vspace{2mm}
\begin{proof}[Proof of (\ref{est-grad-r}) in Lemma \ref{lem-almost-isometry}]
	~
	
	Let $\xi=|X^\top|$. Then $|\mrm{d} \xi^{2\a}|=2\a \xi^{2\a-1} c_\d |\nabla^M r|^2 + \a \xi^{2\a-2} s_\d^2 \left| \mrm{d}\(|\nabla^M r|^2\)\right|$. To estimate $\left|\mrm{d}|\nabla^M r|^2\right|$, we have
	\begin{align*}
	\left|\mrm{d}\(|\nabla^M r|^2\)\right|^2 =\left|\mrm{d}\langle \nabla^N r,\nu\rangle^2 \right|^2
	= & 4\langle \nabla^{N}r,\nu\rangle^2 \sum_{i=1}^{n}\left|\nabla^N_{e_i}\langle \nabla^N r,\nu\rangle \right|^2 \\
	= & 4\langle \nabla^{N}r,\nu\rangle^2 \sum_{i=1}^{n}\left|\Hess^N r(e_i,\nu)+B(e_i,\nabla^M r)\right|^2 \\
	\leq &8 \left[\sum_{i=1}^{n}|\Hess^N r(e_i,\nu)|^2+|B|^2|\nabla^M r|^2 \right].
	\end{align*}
	By Hessian comparison theorem for $K_N\geq \mu$, we have
	\begin{align*}
	\sum_{i=1}^{n}|\Hess^N r (e_i,\nu)|^2 \leq |\Hess^N r|^2 \leq n\(\frac{ c_\mu}{ s_\mu}\)^2.
	\end{align*}
	Thus we have
	\begin{align*}
	\left|\mrm{d}\(|\nabla^M r|^2\)\right| \leq 2\sqrt{2n}\(\frac{c_\mu}{ s_\mu}\)+2\sqrt{2}|B||\nabla^M r|,
	\end{align*}
	and hence
	\begin{align*}
	|\mrm{d} \xi^{2\a}| \leq 2\a \xi^{2\a-1}c_\d |\nabla^M r|^2+\a \xi^{2\a-2}s_\d^2\left[ 2\sqrt{2n}\frac{c_\mu}{s_\mu}+2\sqrt{2}|B||\nabla^M r| \right].
	\end{align*}
	By Rauch comparison theorem, we have $1\leq \frac{ s_\d(r)}{ s_\mu(r)} \leq \frac{s_\d(2R)}{ s_\mu(2R)}\leq C(\mu,\d,R)$. By (\ref{main-theo-1.1}), we take $\e_0=\e_0(n,\d,R,A_1)$ sufficiently small such that if (\ref{l2-pinching}) holds with $\e<\e_0$, then $\frac{1}{2h} \leq |X|\leq \frac{3}{2h}$. As $|\nabla^M r|\leq 1$, we deduce that
	\begin{align*}
	|\mrm{d} \xi^{2\a}| \leq &2\a C(n,\mu,\d,R) \xi^{2\a-2} \|X\|_\infty\left[(c_\d+c_{\mu})+ \|X\|_\infty |B|\right] \\
	\leq &2\a C(n,\mu,\d,R) \xi^{2\a-2} \frac{1}{h^2}\left[ ( c_\d+ c_\mu)h+|B| \right] \\
	\leq &2\a C(n,\mu,\d,R) \xi^{2\a-2} \frac{1}{h^2}\( \|H\|_\infty+|B| \),
	\end{align*}
	where we have used (\ref{ineq-2side-bound-ratio-mean-curv}), $ c_\d \leq \max\{1,c_\d(2R)\}$ and $ c_\mu \leq \max\{1,c_\mu(2R)\}$ in the last inequality. Applying Sobolev inequality (\ref{Hoffman-Spruck}) to $f=\xi^{2\a}$ with $\a\geq 1$, we have
	\begin{align*}
	\|\xi\|_{\frac{2\a n}{n-1}}^{2\a} \leq & C(n)|M|^\frac{1}{n}\left[\frac{1}{|M|}\int_M \(|\mrm{d}\xi^{2\a}|+|H|\xi^{2\a}\)\right] \\
	\leq &C(n,\mu,\d,R)|M|^\frac{1}{n} 2\a \frac{1}{h^2}\left[\frac{1}{|M|}\int_M \(\|H\|_\infty+|B|\)\xi^{2\a-2}\right]\\
	\leq &\frac{C(n,\mu,\d,R)}{h^2}2\a|M|^\frac{1}{n}\left[ \|H\|_\infty \|\xi\|_{2\a-2}^{2\a-2}+\|B\|_q \|\xi\|_{\frac{(2\a-2)q}{q-1}}^{2\a-2}\right] \\
	\leq &\frac{C(n,\mu,\d,R,A_1,A_2)}{h^2}2\a \|\xi\|_{\frac{(2\a-2)q}{q-1}}^{2\a-2}.
	\end{align*}
	We take $\tau:=\frac{n}{n-1}\cdot \frac{q-1}{q}$, $\a:=\frac{1}{2}\frac{q-1}{q}c_i+1$, where $c_{i+1}=c_i \tau+\frac{2n}{n-1}$ and $c_0=2$, then
	$$
	c_i=\(2+2\g\)\tau^{i}-2\g\geq 2,  \quad \text{for $i=0,1,\cdots$,}
	$$
	where $\g=\frac{qn}{q-n}$. Since $\frac{2\a n}{n-1}=\(\frac{q-1}{q}c_p+2\)\frac{n}{n-1}=c_i \tau+\frac{2n}{n-1}=c_{i+1}$, we obtain
	\begin{align*}
	\|\xi\|_{c_{i+1}}^{\frac{c_{i+1}}{\tau}}\leq \left[\frac{C(n,\mu,\d,R,A_1,A_2)}{h^2}c_p\right]^\frac{q}{q-1} \|\xi\|_{\frac{(2\a-2)q}{q-1}}^\frac{(2\a-2)q}{q-1}= \left[\frac{C(n,\mu,\d,R,A_1,A_2)}{h^2}c_i\right]^\frac{q}{q-1} \|\xi\|_{c_i}^{c_i}.
	\end{align*}
	and
	\begin{align*}
	\|\xi\|_{c_{i+1}}^{\frac{c_{i+1}}{\tau^{i+1}}}\leq \left[\frac{C(n,\mu,\d,R,A_1,A_2)}{h^2}c_i\right]^{\frac{q}{(q-1)\tau^i}} \|\xi\|_{c_i}^{\frac{c_i}{\tau^i}}.
	\end{align*}
	By iteration, we get
	\begin{align*}
	\|\xi\|_{c_{i+1}}^{\frac{c_{i+1}}{\tau^{i+1}}} \leq \(\prod_{k=0}^{i} c_k^{\frac{1}{\tau^k}}\)^\frac{n}{n-1}\left[ \frac{C(n,\mu,\d,R,A_1,A_2)}{h^2} \right]^{\g\(1-\frac{1}{\tau^{i+1}}\)} \|\xi\|_{c_0}^{c_0}.
	\end{align*}
	Note that $\prod_{k=0}^{\infty} c_k^{\frac{1}{\tau^k}}$ converges to a constant $C(q,n)$ by the ratio test. Since $\frac{c_i}{\tau^i}$ converges to $2+2\g$, together with  (\ref{ineq-2side-bound-ratio-mean-curv}) and (\ref{l2-bound-perp-position}), we have
	\begin{align*}
	\|\xi\|_\infty^{2+2\g} \leq \left[ C(n,q,\mu,\d,R,A_1,A_2)\frac{1}{h^2} \right]^{\g}\|\xi\|_2^2 \leq \frac{C(n,q,\mu,\d,R,A_1,A_2)}{h^{2+2\g}}\e.
	\end{align*}
	As we have $|X|\geq \frac{1}{2h}$, we finally get
	\begin{align*}
	\|\nabla^M r\|_\infty \leq 2h\|\psi\|_\infty \leq C(n,q,\mu,\d,R,A_1,A_2) \e^\frac{1}{2+2\g}.
	\end{align*}
\end{proof}

\section{Examples supporting Conjecture \ref{conj-main}}\label{sec:6}

Via gluing operation between spheres, we construct a family of surfaces $M_\e$, such that for any $q\ge n$, $|M|^{\frac{1}{n}}\|B\|_q$ blows up as $\epsilon\to 0$, but they cannot appear in Main Theorem. 

Let $X_\e: M_\e \ra \R^{n+1}$ be the isometric embedding of $M_\e$. We denote by $H_\e$, $B_\e$ the mean curvature and the 2nd fundamental form of $M_\e$, respectively.

\begin{example}\label{counter-example}
	Let $p$ be a fixed integer. There exists a family of embedded hypersurfaces $\{M_\e(l;p)\}$ of $\R^{n+1}$ which are formed by $p$ spheres with radii close to $1$ by connected sum around $1\le l\le (\log (\epsilon^{-1}))^{\frac{1}{2}}$ points, such that for $\e \ra 0$, the 2nd fundamental form satisfies
	\begin{align}\label{7.1}
	\|B_\e\|_q \to 0 \; (0< q<n), \quad \|B_\e\|_q \ra \infty\; (q>n),  \quad 
	 C_1(n)l\le \|B_\e\|_n \leq C_2(n)l,
	\end{align}
while the mean curvature and first eigenvalue
\begin{align}\label{7.2}
\|H_\e\|_\infty \ra 1, \quad 	\||X_\e|-1\|_\infty\ra 0, \quad \l_1(M_\e(l;p)) \ra 0.
\end{align} 
\end{example}
From an intrinsic viewpoint, the Gromov-Hausdorff limit of $M_\epsilon((\log (\epsilon^{-1}))^{\frac{1}{2}};p)$ is the union of $p$ unit spheres touching each other at infinitely countable many points. 

Example \ref{counter-example} partially supports Conjecture \ref{conj-main} in the sense that $M$ in Main Theorem may not contain any small neck in more general cases.

\begin{rem}
	A family of hypersurfaces glued by spheres in $\mathbb{R}^{n+1}$ whose mean curvature $\|H_\e\|_\infty\le C(n)$ and $\|H_\e\|_2\to 1$ for some fixed $l$ was given in \cite{Aubry-Grosjean} (for detailed construction, see its arXiv version). 	
	Example \ref{counter-example}  improves their examples not only
	to that $\|H_\epsilon\|_\infty\to 1$, and also to the case $l\to \infty$.
	
	In order to keep $\|H_\epsilon\|_\infty\le 1+\e$, we use a catenoid in a suitable scale for the neck of connected sum, which is almost flat near their boundary.
	
	If $l$ is fixed, then by \cite{Aubry-Grosjean} it can be seen that the spectrum of $M_\e$ converges to the disjoint union of spectrums of $p$ unit spheres. This may be still true in Example \ref{counter-example}. Since only $\lambda_1$ is concerned in this paper, instead of the spectrum we will give a direct upper bound of $\lambda_1(M_\epsilon)$ by a suitably chosen test function.
\end{rem}

The construction of Example \ref{counter-example} is as follows. Let us first consider the case $l=1$ and $p=2$. Since the construction can be repeated successively, more general case can be similarly constructed.

Let $S_\e^{+}$ (resp. $S_\e^{-}$) be a compact surface with boundary,
$$x_{n+1}^2=(r_0^{\pm})^2-(r-3\e)^2, \text{where}~ 3\e  \leq r \leq 3\e+r_0^{\pm},$$
where $r_0^{\pm}$ are two constants in $[1-\e,1+\e]$, and $r=\sqrt{x_1^2+x_2^2+\cdots+x_{n}^2}$ is the distance to $x_{n+1}$-axis.
By definition, $S_\e^{\pm}$ is almost isometric to the unit sphere centered at the origin $o$ with a small spherical cap removed.

We join $S_\e^{\pm}$ by a catenoid $N_\e$ such that the $r$-position of $N_\e$ lies in $[\e^2,\e]$, the gluing part from $N_\e$ to $S_\e^{\pm}$ happens at $r\in [\e,3\e]$; see Figure \ref{fig:2} below.
\begin{figure}[htbp]
	\centering
	\includegraphics[width=0.9\textwidth]{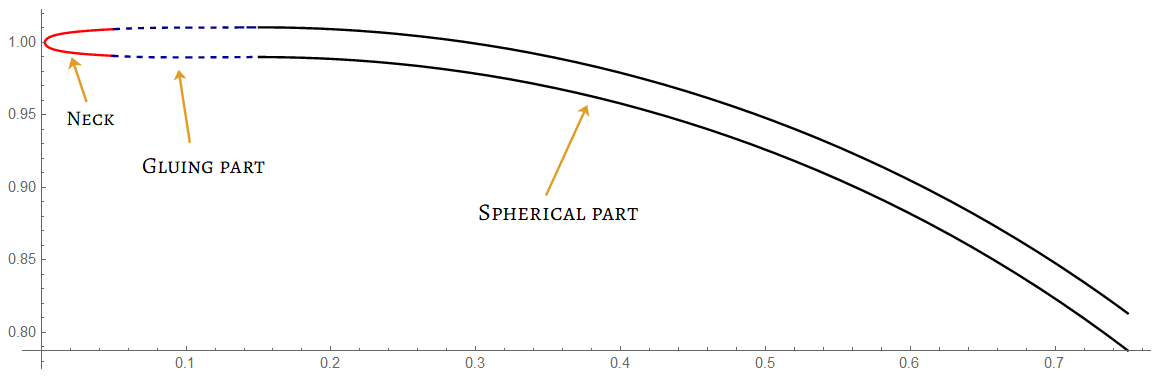}
	\caption{}
	\label{fig:2}
\end{figure}

The $n$-dimensional catenoid in $\mathbb R^{n+1}$ is defined in \cite{doCarmo-Dajczer1983} (cf. \cite{Tam-Zhou2009}). We choose a catenoid fitting in our setting given by
$$x_{n+1}=1\pm\e^2 \psi\(\frac{r}{\e^2}\), \quad \text{where}~\e^2 \leq r \leq \e,$$
and $$\psi(r)=\int_1^{r} \frac{d \tau}{\sqrt{\tau^{2(n-1)}-1}}.$$
When $n=2$, $\psi=\operatorname{arccosh}$.

In the gluing part, we choose polynomials $\rho_{\e,1}^{\pm}$ and $\rho_{\e,2}^{\pm}$, by which the rotational hypersurface smoothly changes from the catenoid to a horizontal hyperplane $P^\pm=\{x_{n+1}=1\pm a_0\}$ for some $a_0=O(\e)$ in $[\e,2\e]$, and then to $S_\e^{\pm}$. They can be chosen as follows
\begin{align*}
x_{3}=&1 \pm a_1 (r-2\e)^3 \pm a_2 (r-2\e)^4 \pm a_0,   \quad \text{for}~\e   \leq r \leq 2\e;\\
x_{3}=&1+b_1^{\pm} (r-2\e)^3+b_2^{\pm} (r-2\e)^4 \pm a_0,   \quad \text{for}~2\e  \leq r \leq 3\e,
\end{align*}
where $a_1$, $a_2$ (resp. $b_1^{\pm}$ and $b_2^{\pm}$) and $a_0$ are determined by the value and derivatives of catenoid $N_\e$ (resp. $S_\e^\pm$) function at $r=\e$ (resp. $r=3\e$) and vanishing $C^1$ and $C^2$ derivatives at $r=2\e$ for the hyperplane. Then $a_0$ is determined by the catenoid $N_\e$, and thus the radius $r_0^\pm$ of $S_\e^{\pm}$ is also determined after $a_0$.

Note that the above operation is based on the fact that the catenoid $N_\e$ can be arbitrary $C^2$-close to a horizontal hyperplane as $\e\to 0$.

Up to an appropriate smoothing, we obtain a family of smooth surfaces $M_\e$ in $\R^{n+1}$. What remains is to verify the conclusion in Example \ref{counter-example}.

Clearly, on $N_\e$ the mean curvature vanishes. By direct calculation, at the gluing part the mean curvature $H_\e$ smoothly varies from $0$ to $\frac{1}{2}+\epsilon$. Note that $S_\e^{\pm}$ is by a $3\e$-shift from a standard sphere, the mean curvature of $S_\e^{\pm}$ is close to $\frac{1}{2}$ at $r=3\e$.  At the $S_\e^{\pm}$ part $r\in [3\e,3\e+r_0^\pm]$, $H_\e=1-\frac{3\e}{2r}+O(\e^2)$. Hence $\|H_\e\|_\infty\to 1$ as $\e\to 0$.

Let $M_\e$ be the hypersurface constructed as above. We now verify $\lambda_1(M_\e)\to 0$. For the bound on the 2nd fundamental form, see the appendix.

We will choose a suitable test function. Let $d(\cdot)$ be the distance function on $M_\e$ to the set $E_\e=\{r=\e^2,x_{n+1}=1\}$.
Then $M_\e$ is split into two parts $M_\e^\pm$ along $E_\e$, and $M_\e^+$ lies outside of $M_\e^-$.

Let us consider a cut-off function from \cite[Propsition 1.3.1]{Courtois1987}, which is defined by
\begin{equation}\label{cut-off-function}
\chi_\e (s)=\left\{ \begin{aligned}
&0, \quad  & 0\leq s \leq \e;\\
&-\frac{2}{\log \e}\log(\frac{s}{\e}), \quad &\e \leq s \leq \sqrt{\e};\\
&1, \quad &s \geq \sqrt{\e}.
\end{aligned} \right.
\end{equation}
The test function is defined to be $f(x):=f_0(x) \cdot \chi_{\e}(d(x))$, where
$$
f_0(x):=\left\{ \begin{aligned}
&1, \quad  &x\in M_\e^{+},\\
&-1,   \quad  &x \in  M_\e^{-}.
\end{aligned} \right.
$$

By definition, $f$ is a smooth function on $M_\e$ such that $f\equiv 1$ on $ M_e^{+} \cap \{d\geq \sqrt{\e}\}$ and $f\equiv -1$ on $M_e^{-} \cap \{d\geq \sqrt{\e}\}$. As the area of $\{d\leq \sqrt{\e}\}$ tends to $0$ as $\e \to 0$, we have
\begin{align*}
\overline{f}=\frac{1}{|M_\e|}\int_{M_\e} f \to 0, \qquad \text{as $\e \to 0$}.
\end{align*}
So we have
\begin{align*}
\int_{M_\e}(f-\overline{f})^2 \to 2\omega_n, \qquad \text{as $\e \to 0$}.
\end{align*}
On the other hand, we have
$$
d(x)=\int_{\e^2}^{r(x)} \sqrt{1+\(\frac{dx_{n+1}}{dr}\)^2}dr  \geq  r(x)-\e^2,
$$
where we recall that $r(x)$ is the Euclidean distance from $x$ to the $x_{n+1}$-axis. Hence, on $\{\e\leq d \leq \sqrt{\e} \}$ we have
$$
r(x) \leq (1+\e)d(x).
$$
It follows that for $n\geq 2$,
\begin{equation}\label{ineq-est-df}
\begin{aligned}
\int_{M_\e}|df|^2 \leq &\frac{8 \omega_{n-1}}{(\log \e)^2}\int_{\e}^{\sqrt{\e}} \frac{r^{n-1}}{s^2}ds 
\\
\leq &\frac{8\omega_{n-1} (1+\e)^{n-1}}{(\log \e)^2}\int_{\e}^{\sqrt{\e}} s^{n-3}ds \to 0, \quad \text{as $\e \ra 0$}.
\end{aligned}
\end{equation}

By the variational characterization of $\l_1(M_\e)$, we conclude
\begin{align*}
\l_1(M_\e)\le \frac{\int_{M_\e}|df|^2}{\int_{M_\e}(f-\overline{f})^2}\to 0.
\end{align*}

Let us consider the case for $1\le l\le (-\log \epsilon)^{1/2}$ and any fixed $p$. By comparing the volume, $O(\epsilon^{-n})$-many necks can be constructed between any two spheres. Then the mean curvature, position vector remains the same as above. By \S \ref{subsec-2nd-fund}, $C_1(n)l\le \|B_\epsilon\|_n\le C_2(n)l$, and 	$\|B_\e\|_q \to 0 \; (0< q<n)$.

By direct calculation, it is easy to see that after multiplying $(-\log \epsilon)^{1/2}$ to (\ref{ineq-est-df}), it still tends to $0$ as $\epsilon\to 0$. Hence $\lambda_1(M_\e)\to 0$.

\section{Appendix}

\subsection{Proofs of technical Lemmas}

\begin{proof}[Proof of Lemma \ref{lem-infity-bound-position}]
	We take $\varphi=|X|=s_\d$ to make a right iteration function.
	For any $\alpha\in \mathbb R$,  $$|\mrm{d}\varphi^{2\a}|\leq 2\a\varphi^{2\a-1} c_\d.$$
	We claim that:
	\begin{claim}
		for all $\d\in\R$,
		$$
		|\mrm{d}\varphi^{2\a}|\leq 2(\sqrt{2}+1)\a \varphi^{2\a-1}\|H\|_\infty\|\varphi\|_\infty.
		$$
	\end{claim}
	Indeed, by (\ref{4.5}) and (\ref{lower-bound-mean-curv})
	$$
	1\leq (\d+\|H\|_\infty^2)\|X\|_2^2\leq (|\d|+\|H\|_\infty^2)\|X\|_2^2 \leq 2\|H\|_\infty^2 \|\varphi\|_\infty^2.
	$$
	Then the claim is true for $\d \ge 0$, for $c_\d\le 2\|H\|_\infty^2 \|\varphi\|_\infty^2$.
	So is for $\d<0$, because by (\ref{lower-bound-mean-curv}) one has
	\begin{align*}
	|\mrm{d}\varphi^{2\a}|\leq &2\a\varphi^{2\a-1}\sqrt{1-\d  s_\d^2}
	\leq 2\a(1+\sqrt{|\d|}\|\varphi\|_\infty)\varphi^{2\a-1}
	\leq 2\a(1+\|H\|_\infty\|\varphi\|_\infty)\varphi^{2\a-1}.
	\end{align*}

	Now by applying Sobolev inequality (\ref{Hoffman-Spruck}) to $f=\varphi^{2\a}$ with $\a\geq 1$, one has
	\begin{align*}
	\|\varphi\|_{\frac{2\a n}{n-1}}^{2\a}\leq & C(n)|M|^\frac{1}{n}\left[(2(\sqrt{2}+1)\a+1)\|H\|_\infty\|\varphi\|_\infty\|\varphi\|_{2\a-1}^{2\a-1} \right]\\
	\leq & C(n)2\a \(|M|^\frac{1}{n}\|H\|_\infty\) \|\varphi\|_\infty\|\varphi\|_{2\a-1}^{2\a-1}.
	\end{align*}
	We take $\tau=\frac{n}{n-1}$ and $\a=\frac{a_i+1}{2}$, where $a_{i+1}=(a_i+1)\tau$ and $a_0=2$, then
	$$
	a_{i}=(2+n)\tau^{i}-n\geq 2, \quad \text{for $i=0,1,\cdots$,}
	$$
	and $\frac{2\a n}{n-1}=2\a \tau=(a_i+1)\tau=a_{i+1}$. Then we obtain
	\begin{align*}
	\|\varphi\|_{a_{i+1}}^{\frac{a_{i+1}}{\tau}} &\leq  C(n)\(|M|^\frac{1}{n}\|H\|_\infty\) (a_i+1) \|\varphi\|_\infty\|\varphi\|_{2\a-1}^{2\a-1}\\
	&\leq C(n)\(|M|^\frac{1}{n}\|H\|_\infty\) a_i \|\varphi\|_\infty\|\varphi\|_{a_i}^{a_i}.
	\end{align*}
	By iteration, we have
	\begin{align*}
	\|\varphi\|_{a_{i+1}}^{\frac{a_{i+1}}{\tau^{i+1}}} &\leq \left[C(n)\(|M|^\frac{1}{n}\|H\|_\infty\) a_i \|\varphi\|_\infty\right]^{\frac{1}{\tau^i}}\|\varphi\|_{a_i}^{\frac{a_i}{\tau^i}} \\
	&\leq \(\prod_{k=0}^{i} a_k^{\frac{1}{\tau^k}}\)\left[ C(n)\(|M|^\frac{1}{n}\|H\|_\infty\)  \|\varphi\|_\infty\right]^{n\(1-\frac{1}{\tau^{i+1}}\)}\|\varphi\|_{a_0}^{a_0}.
	\end{align*}
	Since
	\begin{align*}
	\log\(\prod_{k=0}^{\infty} a_k^{\frac{1}{\tau^k}}\)=\sum_{k=0}^{\infty}\frac{\log a_k}{\tau^k},
	\end{align*}
	which by ratio test $\lim_{k\ra \infty}\left|\frac{\log a_{k+1}}{\tau^{k+1}}\frac{\tau^k}{\log a_k} \right|=\frac{n-1}{n}<1$, the series $\prod_{k=0}^{\infty} a_k^{\frac{1}{\tau^k}}$ converges to a constant $C(n)$. Combining with the fact $\frac{a_i}{\tau^i}$ converges to $2+n$, we obtain
	\begin{align*}
	\|\varphi\|_\infty^{2+n} \leq C(n)\(|M|^\frac{1}{n}\|H\|_\infty\)^n \|\varphi\|_\infty^n \|\varphi\|_2^2\leq C(n)A_1^n \|\varphi\|_\infty^n \|\varphi\|_2^2,
	\end{align*}
	which gives $\|X\|_\infty\leq C(n)A_1^{\frac{n}{2}} \|X\|_2$.
\end{proof}

\vspace{2mm}
\begin{proof}[Proof of (\ref{ineq-upper-bound-Y}) in Lemma \ref{lem-estimate-psi-1}]
	~
	
	By (\ref{3.1}) and $\d s_\d^2+c_\d^2=1$, we have
	\begin{equation}\label{ineq-inner-product}
	\begin{split}
	-\frac{1}{|M|}\int_M c_\d H \langle X,\nu\rangle\leq &\frac{1}{|M|}\int_M c_\d \(\frac{1}{n}\divv_M X^\top-c_\d\)=\frac{\d}{n}\|X^\top\|_2^2-1+\d\|X\|_2^2.
	\end{split}
	\end{equation}
	Then by $L^2$-pinching (\ref{l2-pinching}), (\ref{l2-bound-perp-position}) and (\ref{ineq-2side-bound-ratio-mean-curv}), we have
	\begin{align*}
	\|Y\|_2^2=&\frac{1}{|M|}\int_M  \(H^2 c_\d^2-2\|H\|_\infty^2 \int_M c_\d H \langle X,\nu\rangle+\|H\|_\infty^4 |X|^2 \) \\
	\leq & \frac{2\d\|H\|_\infty^2}{n}\|X^\top\|_2^2-\|H\|_\infty^2(1-\d\|X\|_2^2)+\|H\|_\infty^4 \|X\|_2^2 \\
	\leq & \frac{4|\d|}{n}\e+\|H\|_\infty^2\left[ (\d+\|H\|_\infty^2)\|X\|_2^2-1\right] \\
	\leq & C(n)\|H\|_\infty^2\e \\
	\leq &C(n,\d,R)h^2\e.
	\end{align*}
\end{proof}

In order to prove (\ref{ineq-upper-bound-W-final}) for $\|W\|_2$, let us make some preparation.

\begin{lem}
	\begin{align}
	&\|W\|_2^2 \leq C(n) \(\|H\|^2_\infty\|X\|_\infty+h\)\e \label{ineq-est-W}
	\end{align}
\end{lem}
\begin{proof}
	By (\ref{ineq-inner-product}) we have
	\begin{align*}
	\|W\|_2^2=&\frac{1}{|M|}\int_M  \(|X||\d X+H c_\d \nu|^2-2h\langle \d X+H c_\d\nu,X\rangle+h^2|X|\) \\
	\leq &\frac{1}{|M|}\int_M  |X||\d X+H c_\d \nu|^2 -2h\d \|X\|_2^2-\frac{2h}{|M|}\int_M  c_\d H\langle X,\nu\rangle +h^2\|X\|_2 \\
	\leq &\frac{1}{|M|}\int_M  |X||\d X+H c_\d \nu|^2 -2h\d \|X\|_2^2+\frac{2h\d}{n}\|X^\top\|_2^2-2h(1-\d\|X\|_2^2)+h^2\|X\|_2\\
	\leq & \frac{1}{|M|}\int_M  |X||\d X+H c_\d \nu|^2 -2h+\frac{4h}{n}\e+h^2\|X\|_2.
	\end{align*}
	To estimate the first term, we calculate
	\begin{align*}
	\frac{1}{|M|}\int_M  |X||\d X+H c_\d \nu|^2 =&\frac{1}{|M|}\int_M  |X|\(\d^2 |X|^2+2\d H\langle X,\nu\rangle  c_\d+H^2  c_\d^2\) \\
	= &\frac{1}{|M|}\int_M  |X|\left[ \d+H^2 -\d(H^2|X|^2+ c_\d^2-2 H\langle X,\nu\rangle  c_\d) \right] \\
	\leq &h^2 \|X\|_2-\frac{\d}{|M|}\int_M  |X| |HX- c_\d\nu|^2 \\
	\leq &h^2 \|X\|_2+\frac{|\d|\|X\|_\infty}{|M|}\int_M  |HX- c_\d\nu|^2,
	\end{align*}
	and
	\begin{align*}
	\frac{1}{|M|}\int_M  |HX- c_\d \nu|^2 = &\frac{1}{|M|}\int_M  \(H^2|X|^2-2 c_\d H\langle X,\nu\rangle+c_\d^2\)\\
	\leq & (\d+\|H\|_\infty^2)\|X\|_2^2-1+\frac{2\d}{n}\|X^\top\|_2^2\\
	\leq & \e+\frac{4|\d|}{n\|H\|_\infty^2}\e\leq C(n)\e.
	\end{align*}
	Finally, together with (\ref{lower-bound-mean-curv}), we obtain
	\begin{align*}
	\|W\|_2^2 \leq & C(n)|\d|  \|X\|_\infty\e-2h+\frac{4h}{n}\e+2h^2\|X\|_2. \\
	\leq &C(n)|\d| \|X\|_\infty\e+2h(\sqrt{1+\e}-1)+\frac{4h}{n}\e\\
	\leq &C(n) \(\|H\|^2_\infty\|X\|_\infty+h\)\e.
	\end{align*}
\end{proof}

We now are ready to bound $\|W\|_2$ as follows.
\vspace{2mm}
\begin{proof}[Proof of (\ref{ineq-upper-bound-W-final}) in Lemma \ref{lem-estimate-psi-1}]
	~
	
	By (\ref{ineq-upper-bound-infinity-position}), combining with (\ref{l2-pinching}) and (\ref{ineq-2side-bound-ratio-mean-curv}), we derive
	\begin{align*}
	\|H\|^2_\infty\|X\|_\infty \leq \|H\|_\infty^2 C(n)A_1^{\frac{n}{2}}\|X\|_2\leq C(n,\d,R)h^2 A_1^{\frac{n}{2}}\frac{\sqrt{1+\e}}{h} \leq C(n,\d,R,A_1)h.
	\end{align*}
	Then it follows from (\ref{ineq-est-W}) that
	\begin{align}\label{4.13}
	\|W\|_2^2 \leq C(n)\e(h+\|H\|_\infty^2 \|X\|_\infty) \leq C(n,\d,R,A_1)h \e.
	\end{align}
	
\end{proof}

\subsection{Bounds on 2nd fundamental form for Example \ref{counter-example}}\label{subsec-2nd-fund}

For simplicity, we only give the detailed formulation for $n=2$. The higher dimensional case can be similarly verified by following the construction in \S \ref{sec:6}.

Let $r$ be the Euclidean distance to the $z$-axis. By the construction in \S \ref{sec:6}, for $0<\e \ll 1$, $M=M_\e^{+}\cup M_\e^{-}$, and
$$M_\e^{\pm}:=\{(x,y,z)\in \mathbb R^3 ~|~x^2+y^2=r^2, z= \rho^{\pm}_\e(r), r\in [\epsilon^2,r_0^{\pm}+3\e]\},$$
is given by two $C^2$-functions $\rho_\e^{\pm}:[\e^2,1+3\e]\ra \R$ defined by
\begin{align*}
\rho^{\pm}_\e(r):=\left\{\begin{aligned}
&1\pm\e^2 \operatorname{arccosh}\(\frac{r}{\e^2}\), \quad &\text{if}~&\e^2 \leq r \leq \e;\\
&1 \pm a_1 (r-2\e)^3 \pm a_2 (r-2\e)^4 \pm a_0,   \quad &\text{if}~&\e   \leq r \leq 2\e;\\
&1+b_1^{\pm} (r-2\e)^3+b_2^{\pm} (r-2\e)^4 \pm a_0,   \quad &\text{if}~&2\e  \leq r \leq 3\e;\\
&\sqrt{(r_0^{\pm})^2-(r-3\e)^2},          \quad &\text{if}~&3\e  \leq r \leq 3\e+r_0^{\pm},
\end{aligned}
\right.
\end{align*}
where the coefficients $a_i$, $b_i^{\pm}$ and $r_0^\pm$ are
\begin{align*}
&a_1=\frac{2-3\e^2}{3\e(1-\e^2)^\frac{3}{2}}, \quad &a_2=&\frac{1-2\e^2}{4\e^2(1-\e^2)^\frac{3}{2}}, \quad &a_0=&a_1\e^3-a_2 \e^4+\e^2 \operatorname{arccosh}(\e^{-1}),\\
&b_1^{\pm}=\frac{1}{3\e r_0^\pm},           \quad &b_2^{\pm}=&-\frac{1}{4\e^2 r_0^\pm}, \quad &r_0^\pm=&\frac{1\pm a_0}{2}+\frac{1}{2}\sqrt{(1\pm a_0)^2+\frac{\e^2}{3}}
\end{align*}
We also have $r_0^{\pm}=1\mp \e^2 \log\e+(\e^2)$, and hence $r_0^{+}>1>r_0^{-}$ if $\e$ is sufficiently small.

In order to estimate 2nd fundamental form, the principal curvatures $\k_i^{+}$(resp. $\k_i^{-}$) of $M_\e^{+}$ (resp. $M_\e^{+}$) are calculated as follows.
\begin{enumerate}[(i)]
	\item For $r \in [\e^2,\e]$,
	\begin{align*}
	\k_1^{\pm}(r)=\frac{\e^2}{r^2}, \quad \k_2^{\pm}(r)=-\frac{\e^2}{r^2}.
	\end{align*}
	
	\item For $r=(1+\tau)\e \in [\e,2\e]$ with $\tau \in [0,1]$,
	\begin{align*}
	\k_1^{\pm}(\tau)=\pm (1+2\tau-3\tau^2)+O(\e^2),\quad \k_2^{\pm}(\tau)=\mp(1-\tau)^2+O(\e^2),
	\end{align*}
	which implies $H_\e=\pm 2\tau(1-\tau)+O(\e^2)$. For $\tau\in [0,1]$, $|2\tau(1-\tau)|\leq \frac{1}{2}$.
	
	\item For $r=(2+\tau)\e\in [2\e,3\e]$ with $\tau \in [0,1]$,
	\begin{align*}
	\k_1^{\pm}(\tau)=(-2\tau+3\tau^2)+O(\e^2), \quad \k_2^{\pm}(\tau)=&\frac{-\tau^2+\tau^3}{2+\tau}+O(\e^2).
	\end{align*}
	which implies  $H_\e=\frac{\tau(-4+3\tau+4\tau^2)}{2(2+\tau)}+O(\e^2)$.
	For $\tau\in [0,1]$,  $|\frac{\tau(-4+3\tau+4\tau^2)}{2(2+\tau)}|\leq \frac{1}{2}$.
	
	\item For $r\in [3\e,3\e+r_0^{\pm}]$,
	\begin{align*}
	\k_1^{\pm}(r)=1+O(\e^2), \quad \k_2^{\pm}(r)=1-\frac{3\e}{r}+O(\e^2).
	\end{align*}
	Hence $H_\e=1-\frac{3\e}{2r}+O(\e^2)$.
\end{enumerate}

By (i), the 2nd fundamental form of the catenoid $N_\e$, $|B_\e|^2=\frac{2\e^4}{r^4}$. Since the induced metric is $g_{N_\e}=\(\frac{r^2}{r^2-\e^4}\)dr^2+r^2 d\t^2$, the volume form $d\mu=\frac{r^2}{\sqrt{r^2-\e^4}} dr d\t$. So we derive
\begin{align*}
\int_{N_\e} |B_\e|^2 d\mu=2\int_{\e^2}^{\e} \int_0^{2\pi} 2\frac{\e^4}{r^4} \frac{r^2}{\sqrt{r^2-\e^4}} dr d\t=8\pi \sqrt{1-\e^2}<+\infty.
\end{align*}

For $q>2$,
\begin{align*}
\int_{N_\e} |B_\e|^q d\mu=&2\int_{\e^2}^{\e} \int_0^{2\pi} \(2\frac{\e^4}{r^4}\)^\frac{q}{2} \frac{r^2}{\sqrt{r^2-\e^4}} dr d\t \\
=&4\pi 2^\frac{q}{2} \e^{4-2q} \int_0^{\operatorname{arccosh}(1/\e)} \cosh^{2-2q}(t) dt \\
\geq &4\pi 2^\frac{q}{2} \e^{4-2q}  \cosh^{2-2q}(1) \ra \infty, \quad \text{as $\e\ra 0$}.
\end{align*}

Similary, for $1\le q<2$, $\|B_\epsilon\|_q=O(\epsilon^{4-2q}\log \e)$ and for $0<q\le 1$, $\|B_\epsilon\|_q=O(\epsilon^2\log \e)$.

By our construction, the part $\{\e\leq r\leq 3\e\}\cap M_\e$ has bounded principal curvatures and small area.

For the part $\{r\geq 3\e\}\cap M_\e$, its principal curvatures $\le 1$, and the area is close to $2 \omega_2$.

Summarizing above, a hypersurface formed by $l$ times gluing as above satisfies $C_1 l\le \|B_\e\|_{2}\leq C_2 l$, 
$\|B_\e\|_q \ra \infty$ $(q>2)$, and for $1\le l\le (-\log\e)^{\frac{1}{2}}$,  $\|B_\e\|_q \to 0$ $(0<q<2)$ as $\e\ra 0$.

\end{document}